\documentclass[unnumsec,webpdf,modern,large]{oup-authoring-template}

\onecolumn 

\graphicspath{{Fig/}}
\newcommand{\societylogo}{}

\theoremstyle{thmstyleone}%
\newtheorem{theorem}{Theorem}
\newtheorem{proposition}[theorem]{Proposition}%
\newtheorem{lemma}[theorem]{Lemma}%
\newtheorem{corollary}[theorem]{Corollary}%
\theoremstyle{thmstyletwo}%
\newtheorem{remark}{Remark}%
\theoremstyle{thmstylethree}%

\usepackage{mathtools}
\usepackage{bm}
\usepackage{mathrsfs}
\usepackage{microtype}
\allowdisplaybreaks[4]
\newcommand{\R}{\mathbb R}
\newcommand{\Sph}{\mathbb S^2}
\providecommand{\norm}[1]{\left\lVert #1\right\rVert}
\providecommand{\abs}[1]{\left\lvert #1\right\rvert}

\newcommand{\Tstar}{T_{\ast}}

\newcommand{\Ik}{\mathcal I}

\hypersetup{
  pdftitle={Continuation and threshold global strong solutions for the two-dimensional Landau--Lifshitz--Gilbert equation},
  pdfauthor={Pang-hung Chung and Dan Han},
  pdfsubject={Two-dimensional Landau--Lifshitz--Gilbert equation},
  pdfkeywords={Landau--Lifshitz--Gilbert equation, continuation criterion, ultraviolet tail, threshold energy}
}
\begin{document}

\journaltitle{xxx}
\DOI{DOI added during production}
\copyrightyear{YEAR}
\pubyear{YEAR}
\vol{XX}
\issue{x}
\access{Published: Date added during production}
\appnotes{Paper}

\firstpage{1}


\title[Continuation and threshold global strong solutions for LLG]{Continuation and threshold global strong solutions for the two-dimensional Landau--Lifshitz--Gilbert equation}

\author[1]{Pang-hung Chung}
\author[2,$\ast$]{Dan Han}

\address[1]{\orgdiv{Department of Applied Mathematics}, \orgname{Guangdong University of Education}, \orgaddress{Guangzhou 510640, P. R. China}}
\address[2]{\orgdiv{Department of Mathematics}, \orgname{University of Louisville}, \orgaddress{Louisville, KY 40245, USA}}

\corresp[$\ast$]{Corresponding author. \href{mailto:dan.han@louisville.edu}{dan.han@louisville.edu}}


\abstract{This paper studies finite-terminal continuation and subthreshold global existence for the two-dimensional Landau--Lifshitz--Gilbert equation in the $H_Q^2$ strong-solution class. The proof shows that if a maximal strong solution terminates at a finite time and the full-frequency $L^4_{t,x}$ norm of $\nabla u$ is finite on a terminal time window, then the strong solution can be continued. Furthermore, the solution can still be continued as long as the ultraviolet tail above one fixed dyadic cutoff belongs to $L^4_{t,x}$. As applications, we obtain global strong solutions when $E(u_0)<4\pi$; in the degree-zero sector, this energy threshold can be raised to $8\pi$.}


\maketitle

\medskip

\section{Introduction}

This paper considers the two-dimensional $\Sph$-valued Landau--Lifshitz--Gilbert equation (abbreviated below as LLG)
\begin{equation}\label{eq:LLG-intro}
\partial_t u
=
 a\bigl(\Delta u+|\nabla u|^2u\bigr)+b\,u\times \Delta u,
\qquad
u:I\times\R^2\to \Sph,
\qquad a>0,
\ b\in\R.
\end{equation}
Here $a>0$ denotes the Gilbert damping strength, and $b\,u\times\Delta u$ is the precessional term. The extreme case $b=0$ degenerates into the harmonic map heat flow, while $a=0$ degenerates into the two-dimensional Schr\"odinger map. For early connections between LLG and the geometry of harmonic maps, see Guo--Hong, and for systematic background, see the monograph \cite{GuoHong,GuoDingBook}; moreover, the purely dissipative endpoint $b=0$ is directly connected with Struwe's classical theory for the two-dimensional harmonic map heat flow \cite{Struwe}. On the dispersive side, the geometric formulation of Schr\"odinger maps developed by Chang--Shatah--Uhlenbeck is a foundational point of comparison for the endpoint $a=0$ \cite{ChangShatahUhlenbeck}. Thus \eqref{eq:LLG-intro} carries both geometric parabolic and geometric dispersive structures. For this reason, it is a natural intermediate testing ground for comparing Schr\"odinger maps with the harmonic map heat flow.

In two dimensions, the scaling
\begin{equation*}
u_\lambda(t,x):=u(\lambda^2 t,\lambda x)
\end{equation*}
preserves both equation \eqref{eq:LLG-intro} and the Dirichlet energy
\begin{equation*}
E(u(t)):=\frac12\int_{\R^2}|\nabla u(t,x)|^2\,dx,
\end{equation*}
and hence the problem is energy-critical. For the classical analytic theory of pure Schr\"odinger maps, see the works of Nahmod--Stefanov--Uhlenbeck and McGahagan \cite{NahmodStefanovUhlenbeck,McGahagan}. A major development in the critical theory was the global small-data result of Bejenaru--Ionescu--Kenig--Tataru in the critical Sobolev spaces for dimensions $d\ge2$ \cite{BejenaruIonescuKenigTataru}. On the energy-critical continuation problem in two dimensions, Dodson--Smith proved that, as long as
\[
\|\nabla\phi\|_{\ell^2L^4_{t,x}}
=
\Biggl(\sum_{k\in\mathbb Z}\|P_k\nabla\phi\|_{L^4_{t,x}}^2\Biggr)^{1/2}
\]
is bounded on the terminal interval, the solution can be continued; here $P_k$ denotes the Littlewood--Paley projection to frequency about $2^k$. This compresses the large-data global theory into a problem of controlling one controlling norm. This point is precisely the starting point of the present paper. The reason Dodson--Smith's argument requires the whole $\ell^2L^4$ block is, in essence, the non-perturbative nature of the near-diagonal high-frequency coupling term $A_\ell\partial_\ell\psi_m$ in the differentiated system. See the main theorem and Section 6 of the original paper \cite{DodsonSmith}. In the broader low-regularity Schr\"odinger-map theory, the differentiated-variable and frequency-localized viewpoint had already played a decisive role in the global well-posedness theory of Ionescu--Kenig in dimensions $d\ge3$ \cite{IonescuKenigII}.

The question addressed here is: \emph{once the object of study is changed to the two-dimensional LLG with the simplest Gilbert damping term, can the above $\ell^2L^4$ condition be substantially relaxed?}

The answer given in this paper is affirmative, and the amount of relaxation is quite large. For the two-dimensional LLG, controlling the whole $\ell^2L^4$ block is not necessary. In fact, one does not even have to track the square-summed full-frequency quantity. The genuine parabolic dissipation produced by damping allows us to work directly at the $H^2$ level by energy methods, thereby reducing the finite-terminal continuation problem to a single critical $L^4$ control. The basic differential inequality proved in this paper is
\begin{equation}\label{eq:intro-H2-ineq}
\frac12\frac{d}{dt}\|\Delta u(t)\|_{L^2_x}^2
+\frac a2\|\nabla\Delta u(t)\|_{L^2_x}^2
\le C_{a,b}\|\nabla u(t)\|_{L^4_x}^4\|\Delta u(t)\|_{L^2_x}^2.
\end{equation}
Its structure is similar to Serrin-type continuation criteria in systems of Navier--Stokes or Ericksen--Leslie type: as long as $\nabla u\in L^4_{t,x}$ on a terminal time window, Gronwall's inequality controls $\|\Delta u\|_{L^2}$ and thereby excludes $H^2$ blow-up. This is analytically different from the undamped Schr\"odinger flow near harmonic maps, where modulation and dispersive mechanisms replace the positive parabolic term; an important model analysis is due to Gustafson--Kang--Tsai \cite{GustafsonKangTsaiCPAM}. Going further, by combining the damped $H^2$ estimate with an automatic dyadic low-frequency budget on finite time windows, one can compress the criterion into a single ultraviolet-tail criterion.

It should be stressed that this relaxation is not a formal game. Finite-time singularity mechanisms for Landau--Lifshitz type flows are well documented, including the finite-time singularity result of Ding--Wang and the refined two-dimensional equivariant blow-up dynamics of Xu--Zhao. At the purely dispersive endpoint, Merle--Rapha\"el--Rodnianski constructed finite-time blow-up dynamics for smooth equivariant solutions of the critical Schr\"odinger-map problem \cite{MerleRaphaelRodnianski}. Therefore, a criterion that truly identifies which kinds of quantities must diverge if a solution blows up in finite time has independent mathematical significance. Related published LLG results may be found in \cite{DingWang,XuZhao}.

In the presentation, this paper deliberately avoids simply replacing the caloric-gauge proof for Schr\"odinger maps by a ``damped version'' formulation. Instead, the argument starts from the geometric identities, derives parabolic $H^2$ a priori estimates, uses these estimates to establish finite-terminal continuation under full-frequency $L^4$ control, and finally reduces the continuation criterion to a single ultraviolet-tail condition. The dependence of stability, concentration, and oscillation mechanisms on the passage from harmonic map heat flow through Landau--Lifshitz dynamics to Schr\"odinger maps was analyzed in detail by Gustafson--Nakanishi--Tsai \cite{GustafsonNakanishiTsai}. Thus the present work is not a slight repair appended to the paper of Dodson--Smith; rather, using the additional dissipation of LLG, it rewrites the continuation mechanism itself in a form that is closer to the intrinsic nature of parabolic equations.

We finally explain the relation between this paper and the existing literature. At the weak-solution level, Alouges--Soyeur long ago constructed global weak solutions for LLG and exhibited non-uniqueness, while Di~Fratta--Innerberger--Praetorius proved weak--strong uniqueness in the framework of micromagnetics \cite{AlougesSoyeur,DiFrattaInnerbergerPraetorius}. For small-data global theory, Lin--Lai--Wang established global well-posedness for LLG in Morrey spaces, while Guti\'errez--de~Laire studied the Cauchy problem and self-similar solutions in a small $\mathrm{BMO}$ framework \cite{LinLaiWang,GutierrezDeLaire}. On the Schr\"odinger-map side, Bejenaru--Ionescu--Kenig established global existence and uniqueness for small critical data in dimensions $d\ge4$ \cite{BejenaruIonescuKenig}. For bubbling and singularity analysis of the two-dimensional LLG and related Landau--Lifshitz flows, see Harpes, Ding--Wang, and Xu--Zhao \cite{Harpes,DingWang,XuZhao}. For strong solutions and global regularity in higher dimensions, see Melcher; the standard $H^2$ local theory invoked in the continuation step of the present paper may be read together with Carbou--Fabrie, Wang--Guo, and the systematic treatment of Guo--Ding \cite{Melcher,CarbouFabrie,WangGuo,GuoDingBook}. For dynamics near nontrivial harmonic maps at the Schr\"odinger endpoint, Gustafson--Kang--Tsai proved asymptotic stability results in the equivariant setting \cite{GustafsonKangTsaiDuke}. On the other hand, \cite{LiZhao} established a subthreshold global theory for two-dimensional compact targets: when the target is $\Sph$ and $E(u_0)<4\pi$, its unique global $H$-solution satisfies a global $L^4_{t,x}$ control and converges to a constant in the energy space \cite{LiZhao}. The present paper is concerned with a different issue: in the standard strong-solution class, we compress the continuation criterion as far as possible until only one genuinely necessary terminal ultraviolet quantity remains. Once this criterion is combined with the subthreshold global $L^4$ theory of \cite{LiZhao}, Theorem \ref{thm:below4pi-global} follows automatically.

In the standard strong-solution framework, fix a constant vector $Q\in\Sph$ and write
\[
H_Q^m(\R^2;\Sph)
:=
\bigl\{u:\R^2\to\Sph:\ u-Q\in H^m(\R^2;\R^3)\bigr\}.
\]
We adopt the following convention for strong solutions. After fixing initial data $u_0\in H_Q^2(\R^2;\Sph)$, the standard local well-posedness theory gives a maximal existence time $\Tstar\in(0,\infty]$ and a unique $H_Q^2$ strong solution; on every compact subinterval $[0,T]\Subset[0,\Tstar)$ it satisfies
\[
u-Q\in C([0,T];H^2(\R^2;\R^3))
\cap L^2((0,T);H^3(\R^2;\R^3)),
\qquad
\partial_t u\in L^2((0,T);H^1(\R^2;\R^3)).
\]
If $\Tstar<\infty$, then this maximal time also satisfies the standard continuation alternative
\begin{equation*}
\limsup_{t\uparrow \Tstar}\norm{u(t)-Q}_{H^2}=+\infty.
\end{equation*}
Results of this kind go back to Carbou--Fabrie, Melcher, Wang--Guo, and the systematic account of Guo--Ding \cite{CarbouFabrie,Melcher,WangGuo,GuoDingBook}.

The first main result of this paper states that, for the damped two-dimensional LLG, finite-terminal continuation does not require the whole $\ell^2L^4$ controlling norm used by Dodson--Smith for two-dimensional Schr\"odinger maps; finiteness of the full-frequency $L^4$ norm on a terminal time window is already sufficient. More precisely, we prove the following theorem.

\begin{theorem}[Full-frequency continuation]\label{thm:L4-cont}
Let $u$ be a maximal strong solution of \eqref{eq:LLG-intro}, and suppose that
\[
0<\Tstar<\infty.
\]
If there exists $T_0\in[0,\Tstar)$ such that
\begin{equation}\label{eq:L4-ass}
\int_{T_0}^{\Tstar}\int_{\R^2}|\nabla u(t,x)|^4\,dxdt<\infty,
\end{equation}
then
\begin{equation*}
\sup_{T_0\le t<\Tstar}\norm{u(t)-Q}_{H^2}<\infty.
\end{equation*}
Consequently, $u$ can be continued beyond $t=\Tstar$, contradicting the maximality of $\Tstar$.
\end{theorem}

Theorem \ref{thm:L4-cont} shows that the finite-terminal continuation problem for the damped two-dimensional LLG can already be compressed to a full-frequency $L^4$ control. Its contrapositive says that, once finite-time blow-up occurs, the full-frequency $L^4$ quantity must diverge on the terminal time window. But the present paper does not stop at this level. Using the automatic low-frequency budget on finite time windows, one can further compress the full-frequency criterion into an ultraviolet criterion involving only one fixed high-frequency tail. This is the point that is closest to the spirit of the Dodson--Smith controlling norm while at the same time reflecting a genuine improvement brought by damping.

\begin{theorem}[Ultraviolet-tail continuation]\label{thm:uv-tail-cont}
Let $u$ be a maximal strong solution of \eqref{eq:LLG-intro}, and suppose that $0<\Tstar<\infty$. If there exist a terminal interval
\[
I=(T_0,\Tstar)\subset [0,\Tstar)
\]
and a dyadic cutoff $k_\ast\in\mathbb Z$ such that
\begin{equation}\label{eq:uv-tail-ass}
\norm{P_{\ge k_\ast}\nabla u}_{L^4_{t,x}(I\times\R^2)}<\infty,
\end{equation}
then
\begin{equation*}
\sup_{t\in I}\norm{u(t)-Q}_{H^2}<\infty,
\end{equation*}
and hence $u$ can be continued beyond $t=\Tstar$.
\end{theorem}

Theorem \ref{thm:uv-tail-cont} shows that, for the two-dimensional LLG with Gilbert damping, the terminal high-frequency information that truly has to be tracked is not the whole $\ell^2L^4$ block, but only the ultraviolet tail above one fixed cutoff.

The third main conclusion of the paper states that, once the finite-terminal continuation criteria above are combined with the existing subthreshold global theory for the two-dimensional LLG, one obtains a global existence result in the standard strong-solution class.

\begin{theorem}[Subthreshold globality]\label{thm:below4pi-global}
Assume that the initial datum $u_0\in H_Q^2(\R^2;\Sph)$ satisfies
\begin{equation*}
E(u_0)<4\pi.
\end{equation*}
Then the maximal strong solution of \eqref{eq:LLG-intro} actually satisfies
\[
\Tstar=\infty.
\]
Equivalently, the maximal solution has no finite terminal time. Therefore, for every $T<\infty$, its strong-solution regularity persists on $[0,T]$ in the form
\[
u-Q\in C([0,T];H^2(\R^2;\R^3))\cap L^2((0,T);H^3(\R^2;\R^3)),
\qquad
\partial_tu\in L^2((0,T);H^1(\R^2;\R^3)),
\]
and it also satisfies the global spacetime control
\begin{equation*}
\int_0^{\infty}\int_{\R^2}|\nabla u(t,x)|^4\,dxdt<\infty.
\end{equation*}
\end{theorem}

Theorem \ref{thm:below4pi-global} gives the subthreshold global strong-solution conclusion that this paper can close rigorously in all topological sectors; it relies on the two-dimensional subthreshold global theory and therefore naturally stops at $4\pi$. However, in the degree-zero sector, by combining the finite-terminal continuation criterion established here with the classical bubble-profile decomposition and degree quantization, one can push the energy obstruction further up to $8\pi$. This gives the fourth main theorem of the paper.

To state this improvement, we first fix in the introduction the topological notation needed for the degree-zero sector. For
\[
v\in H_Q^1(\R^2;\Sph),
\]
define its topological degree by
\begin{equation}\label{eq:degree-def}
\operatorname{deg}(v)
:=
\frac{1}{4\pi}\int_{\R^2} v\cdot(\partial_1 v\times \partial_2 v)\,dx.
\end{equation}
Since
\[
\bigl|v\cdot(\partial_1 v\times \partial_2 v)\bigr|
\le \frac12 |\nabla v|^2,
\]
the integral above converges absolutely; moreover, for every $v\in H_Q^1(\R^2;\Sph)$ one has the topological energy lower bound
\[
4\pi\,|\deg(v)|\le E(v).
\]
With this notation, we also obtain the following improved global result in the degree-zero sector.

\begin{theorem}[Degree-zero $8\pi$ globality]
\label{thm:deg0-8pi-global}
Assume that $u_0\in H_Q^2(\R^2;\Sph)$ satisfies
\begin{equation}\label{eq:deg0-8pi-main-ass}
\deg(u_0)=0,
\qquad
E(u_0)<8\pi.
\end{equation}
Then the maximal strong solution of \eqref{eq:LLG-intro} satisfies
\[
\Tstar=\infty.
\]
Furthermore, the globality here is obtained in the original strong-solution sense: for every prescribed finite time $T$,
\[
u-Q\in C([0,T];H^2(\R^2;\R^3))\cap L^2((0,T);H^3(\R^2;\R^3)),
\qquad
\partial_tu\in L^2((0,T);H^1(\R^2;\R^3)).
\]
\end{theorem}

The difference between Theorem \ref{thm:deg0-8pi-global} and Theorem \ref{thm:below4pi-global} is this: the latter uses the two-dimensional subthreshold global theory to upgrade a continuation criterion into a global theorem; the former splices together the analytic continuation mechanism established in the first four sections of this paper with bubble quantization under the degree-zero constraint. In other words, the $8\pi$ improvement does not come from a new $H^2$ differential inequality, but from ``the strong-solution continuation framework of this paper plus the geometric energy obstruction under the degree-zero constraint.''

It should be emphasized that the number $4\pi$ is not an internal assumption of the continuation argument in this paper. What is truly established internally here is the finite-terminal full-frequency $L^4$ continuation criterion and the single ultraviolet-tail continuation criterion. The energy thresholds appearing in the two global conclusions of the paper reappear only at the last step, when the two-dimensional subthreshold global theory and the degree-zero bubble quantization are brought in. In this way, the relation between the present paper and Dodson--Smith becomes quite clear: the former uses the parabolic dissipation brought by damping to lower the two-dimensional continuation criterion significantly from $\ell^2L^4$ to full-frequency $L^4$, and then further to a single ultraviolet tail; the latter closes the whole controlling norm in the dispersive framework of the pure Schr\"odinger map by means of the caloric gauge \cite{DodsonSmith,LiZhao}.

The remainder of the paper is organized as follows. Section \ref{sec:prelim} collects the geometric identities, energy controls, and the interpolation and low-frequency budgets needed later. Section \ref{sec:mainproof} establishes the core $H^2$ differential inequality and its Gronwall closure. Section \ref{sec:proofs} proves Theorems \ref{thm:L4-cont}--\ref{thm:below4pi-global} in order. Section \ref{sec:deg0} proves Theorem \ref{thm:deg0-8pi-global} and explains how, in the degree-zero sector, the general $4\pi$ energy obstruction can be pushed up to $8\pi$.

\section{Preliminary estimates}\label{sec:prelim}

This section collects three basic structures that will be used repeatedly below: the extrinsic form, the energy dissipation, and the relative $L^2$ control.

\subsection{Geometric identities and energy control}

Let $u:I\times\R^2\to\Sph$ be sufficiently smooth. From $|u|^2\equiv1$ one obtains
\begin{equation}\label{eq:constraint1}
u\cdot \partial_j u=0,
\qquad
u\cdot \Delta u = -|\nabla u|^2.
\end{equation}
Therefore the tension field
\begin{equation}\label{eq:tension}
\tau(u):=\Delta u+|\nabla u|^2u
\end{equation}
is always tangent, that is, $u\cdot \tau(u)=0$. Hence equation \eqref{eq:LLG-intro} can be written compactly as
\begin{equation}\label{eq:LLG-tau}
\partial_t u = a\tau(u)+b\,u\times\tau(u).
\end{equation}
On the other hand, since
\[
\partial_j(u\times \partial_j u)=\partial_j u\times \partial_j u + u\times \partial_j\partial_j u
=u\times \Delta u,
\]
one may rewrite \eqref{eq:LLG-intro} in divergence form as
\begin{equation}\label{eq:LLG-div}
\partial_t u-a\Delta u
= b\,\partial_j(u\times\partial_j u)+a|\nabla u|^2u.
\end{equation}
This formula is the starting point for the higher-order energy estimates below. In the present damped problem we retain the direct extrinsic formulation, because the term $a\Delta u$ supplies the coercive parabolic contribution used below.

We first record the energy dissipation identity.

\begin{lemma}[Energy law]\label{lem:energy}
Let $u$ be a strong solution of \eqref{eq:LLG-intro}. Then for every $0\le t_0<t_1<\Tstar$,
\begin{equation}\label{eq:energy-id}
E\bigl(u(t_1)\bigr)+a\int_{t_0}^{t_1}\int_{\R^2}|\tau(u)(t,x)|^2\,dxdt
=E\bigl(u(t_0)\bigr).
\end{equation}
In particular, $E(u(t))$ is nonincreasing.
\end{lemma}

\begin{proof}
By the strong-solution regularity $u-Q\in C([0,\Tstar);H^2)$ and $\partial_tu\in L^2_{\mathrm{loc}}([0,\Tstar);H^1)$, the pairings, integrations by parts, and chain rule below are justified. By integration by parts,
\[
\frac{d}{dt}E(u(t))
=
\int_{\R^2}\nabla u:\nabla u_t\,dx
=-\int_{\R^2}\Delta u\cdot u_t\,dx.
\]
Using \eqref{eq:tension} and $u_t\cdot u=0$, we get
\[
-\Delta u\cdot u_t
=-(\tau(u)-|\nabla u|^2u)\cdot u_t
=-\tau(u)\cdot u_t.
\]
Substituting \eqref{eq:LLG-tau}, and observing that
\[
\tau(u)\cdot (u\times \tau(u))=0,
\]
we obtain
\[
\frac{d}{dt}E(u(t))=-a\int_{\R^2}|\tau(u)|^2\,dx.
\]
Integrating in time gives \eqref{eq:energy-id}.
\end{proof}

The following pointwise identity connects $L^4$ control with the $H^2$ energy.

\begin{lemma}[Projection identity]\label{lem:L4-by-Lap}
For every $u:\R^2\to\Sph$, one has the pointwise identity
\begin{equation}\label{eq:uDelta}
|\nabla u|^2 = -u\cdot \Delta u,
\end{equation}
and consequently
\begin{equation*}
|\nabla u|^4 \le |\Delta u|^2
\qquad\text{a.e. on }\R^2.
\end{equation*}
Therefore
\begin{equation}\label{eq:L4-by-Delta-int}
\norm{\nabla u}_{L^4_x}^4\le \norm{\Delta u}_{L^2_x}^2.
\end{equation}
\end{lemma}

\begin{proof}
The first formula is exactly the second identity in \eqref{eq:constraint1}. By Cauchy--Schwarz,
\[
|\nabla u|^4=|u\cdot \Delta u|^2\le |u|^2|\Delta u|^2=|\Delta u|^2.
\]
Integrating yields \eqref{eq:L4-by-Delta-int}.
\end{proof}

We next give the relative $L^2$ control.

To connect with the local strong-solution class $H_Q^2$, we also need to control the $L^2$ norm of $u-Q$.

\begin{lemma}[Relative $L^2$ growth]\label{lem:L2-relative}
Fix $Q\in\Sph$. For every strong solution $u$, one has
\begin{equation}\label{eq:L2-growth}
\frac12\frac{d}{dt}\norm{u(t)-Q}_{L^2_x}^2
\le 2aE\bigl(u(t)\bigr)
\le 2aE\bigl(u_0\bigr).
\end{equation}
Consequently, for every $t\in[0,\Tstar)$,
\begin{equation}\label{eq:L2-growth-int}
\norm{u(t)-Q}_{L^2_x}^2
\le
\norm{u_0-Q}_{L^2_x}^2+4aE(u_0)t.
\end{equation}
\end{lemma}

\begin{proof}
By the strong-solution regularity $u-Q\in C([0,\Tstar);H^2)$ and $\partial_tu\in L^2_{\mathrm{loc}}([0,\Tstar);H^1)$, the function $t\mapsto \|u(t)-Q\|_{L^2_x}^2$ is absolutely continuous, and the pairings and integrations by parts below are justified. Set $v:=u-Q$. Then
\[
\frac12\frac{d}{dt}\norm{v}_{L^2_x}^2
=
\int_{\R^2}v\cdot u_t\,dx
=:J_1+J_2+J_3,
\]
where
\[
J_1=a\int v\cdot \Delta u,
\qquad
J_2=a\int |\nabla u|^2 v\cdot u,
\qquad
J_3=b\int v\cdot (u\times \Delta u).
\]
For the first term, integration by parts gives
\[
J_1=-a\int |\nabla u|^2\,dx=-2aE(u(t)).
\]
For the second term, using $|u|=|Q|=1$, we have $v\cdot u =1-Q\cdot u\le 2$, whence
\[
J_2\le 2a\int |\nabla u|^2\,dx=4aE(u(t)).
\]
The last term is eliminated by the divergence form:
\[
J_3
=-b\int Q\cdot (u\times \Delta u)
=-b\int Q\cdot \partial_j(u\times \partial_j u)
=0.
\]
Combining these estimates yields
\[
\frac12\frac{d}{dt}\norm{u-Q}_{L^2_x}^2
\le 2aE(u(t)).
\]
Energy monotonicity then gives \eqref{eq:L2-growth} and \eqref{eq:L2-growth-int}.
\end{proof}

This section now establishes two types of technical tools needed later. The first type consists of two-dimensional Ladyzhenskaya/Gagliardo--Nirenberg inequalities; the second type is the automatic low-frequency boundedness on terminal time windows. The latter is precisely the bridge that compresses the ``full-frequency $L^4$ continuation criterion'' into the ``single ultraviolet-tail continuation criterion.''

\subsection{Interpolation tools and low-frequency budgets}\label{sec:aux}

\begin{lemma}[Ladyzhenskaya]\label{lem:lady}
For every $f\in H^1(\R^2)$, one has
\begin{equation}\label{eq:lady}
\norm{f}_{L^4_x}^2\le C\norm{f}_{L^2_x}\norm{\nabla f}_{L^2_x}.
\end{equation}
\end{lemma}

\begin{proof}
Let $g:=|f|^2$. The two-dimensional Sobolev embedding $W^{1,1}(\R^2)\hookrightarrow L^2(\R^2)$ gives
\[
\norm{f}_{L^4_x}^2=\norm{g}_{L^2_x}\le C\norm{\nabla g}_{L^1_x}.
\]
On the other hand,
\[
\nabla g=\nabla(|f|^2)=2 f\cdot \nabla f,
\]
and hence, by H\"older's inequality,
\[
\norm{\nabla g}_{L^1_x}\le 2\norm{f}_{L^2_x}\norm{\nabla f}_{L^2_x}.
\]
Combining the two inequalities gives
\[
\norm{f}_{L^4_x}^2\le C\norm{f}_{L^2_x}\norm{\nabla f}_{L^2_x}.
\]
This is \eqref{eq:lady}.
\end{proof}

\begin{lemma}[Interpolation]\label{lem:derived-GN}
Assume $u-Q\in H^3(\R^2)$. Then there exists an absolute constant $C>0$ such that
\begin{align}
\norm{\nabla^2 u}_{L^4_x}^2
&\le C\norm{\Delta u}_{L^2_x}\norm{\nabla\Delta u}_{L^2_x},
\label{eq:nabla2-L4}
\\
\norm{\Delta u}_{L^4_x}^2
&\le C\norm{\Delta u}_{L^2_x}\norm{\nabla\Delta u}_{L^2_x},
\label{eq:Delta-L4}
\\
\norm{\nabla u}_{L^6_x}^3
&\le C\norm{\nabla u}_{L^4_x}^2\norm{\Delta u}_{L^2_x}.
\label{eq:nabla-L6}
\end{align}
\end{lemma}

\begin{proof}
For \eqref{eq:nabla2-L4}, apply Lemma \ref{lem:lady} to $f=\nabla^2u$ and obtain
\[
\norm{\nabla^2 u}_{L^4_x}^2
\le C\norm{\nabla^2 u}_{L^2_x}\norm{\nabla^3 u}_{L^2_x}.
\]
On the whole space, Plancherel's identity gives
\[
\norm{\nabla^2 u}_{L^2_x}=\norm{\Delta u}_{L^2_x},
\qquad
\norm{\nabla^3 u}_{L^2_x}=\norm{\nabla\Delta u}_{L^2_x},
\]
and hence \eqref{eq:nabla2-L4} follows.

For \eqref{eq:Delta-L4}, apply Lemma \ref{lem:lady} directly to $f=\Delta u$:
\[
\norm{\Delta u}_{L^4_x}^2
\le C\norm{\Delta u}_{L^2_x}\norm{\nabla\Delta u}_{L^2_x}.
\]
This is \eqref{eq:Delta-L4}.

For \eqref{eq:nabla-L6}, first use the two-dimensional Gagliardo--Nirenberg interpolation inequality
\[
\norm{f}_{L^6_x}\le C\norm{f}_{L^4_x}^{2/3}\norm{\nabla f}_{L^2_x}^{1/3},
\qquad f\in H^1(\R^2),
\]
and then set $f=\nabla u$. This gives
\[
\norm{\nabla u}_{L^6_x}^3
\le C\norm{\nabla u}_{L^4_x}^2\norm{\nabla^2 u}_{L^2_x}.
\]
Finally, Plancherel's identity $\norm{\nabla^2 u}_{L^2_x}=\norm{\Delta u}_{L^2_x}$ gives
\[
\norm{\nabla u}_{L^6_x}^3
\le C\norm{\nabla u}_{L^4_x}^2\norm{\Delta u}_{L^2_x},
\]
which is \eqref{eq:nabla-L6}.
\end{proof}

We next give two exact differential expansions that will be used in the higher-order energy estimates.

\begin{lemma}[Differential identities]\label{lem:pointwise-expansion}
Let $u$ be sufficiently smooth. For every $j,k\in\{1,2\}$, one has
\begin{align}
\partial_j\bigl(u\times \partial_j u\bigr)
&=
u\times \partial_{jj}u,
\label{eq:cross-div-exact}
\\
\partial_k\partial_j\bigl(u\times \partial_j u\bigr)
&=
(\partial_k u)\times \partial_{jj}u
+
u\times \partial_k\partial_{jj}u,
\label{eq:cross-diff-exact}
\\
\partial_k\bigl(|\nabla u|^2u\bigr)
&=
2\sum_{m=1}^2
\bigl(\partial_m u\cdot \partial_{km}u\bigr)u
+
|\nabla u|^2\partial_k u.
\label{eq:cubic-diff-exact}
\end{align}
Therefore
\begin{align}
\abs{\partial_k\partial_j(u\times \partial_j u)-u\times \partial_k\partial_{jj}u}
&\le C\abs{\nabla u}\abs{\nabla^2u},
\label{eq:point-cross}
\\
\abs{\nabla\bigl(|\nabla u|^2u\bigr)}
&\le C\abs{\nabla u}\abs{\nabla^2u}+C\abs{\nabla u}^3.
\label{eq:point-cubic}
\end{align}
\end{lemma}

\begin{proof}
The first formula follows directly from the product rule and antisymmetry:
\[
\partial_j(u\times \partial_j u)
=
(\partial_j u)\times(\partial_j u)+u\times \partial_{jj}u
=
u\times \partial_{jj}u.
\]
Applying $\partial_k$ to this formula gives
\[
\partial_k\partial_j(u\times \partial_j u)
=
(\partial_k u)\times \partial_{jj}u+u\times \partial_k\partial_{jj}u,
\]
which is \eqref{eq:cross-diff-exact}.

For \eqref{eq:cubic-diff-exact}, write first
\[
|\nabla u|^2=\sum_{m=1}^2 \partial_m u\cdot \partial_m u.
\]
Then
\[
\partial_k(|\nabla u|^2)
=
2\sum_{m=1}^2 \partial_m u\cdot \partial_{km}u,
\]
and the product rule gives
\[
\partial_k\bigl(|\nabla u|^2u\bigr)
=
\partial_k(|\nabla u|^2)\,u+|\nabla u|^2\partial_k u
=
2\sum_{m=1}^2
\bigl(\partial_m u\cdot \partial_{km}u\bigr)u
+
|\nabla u|^2\partial_k u.
\]

Finally, \eqref{eq:cross-diff-exact} immediately implies
\[
\abs{\partial_k\partial_j(u\times \partial_j u)-u\times \partial_k\partial_{jj}u}
=
\abs{(\partial_k u)\times \partial_{jj}u}
\le C\abs{\nabla u}\abs{\nabla^2u},
\]
which is \eqref{eq:point-cross}. Meanwhile \eqref{eq:cubic-diff-exact} gives
\[
\abs{\partial_k\bigl(|\nabla u|^2u\bigr)}
\le
2\sum_{m=1}^2
\abs{\partial_m u}\,\abs{\partial_{km}u}\,\abs{u}
+
|\nabla u|^2\abs{\partial_k u}.
\]
Since $|u|=1$, summing over $k=1,2$ gives \eqref{eq:point-cubic}.
\end{proof}

\begin{lemma}[Laplacian expansion]\label{lem:lap-expansion}
Let $u$ be sufficiently smooth. Then
\begin{align}
\sum_{j=1}^2\Delta\partial_j\bigl(u\times \partial_j u\bigr)
&=
u\times \Delta^2 u
+
2\sum_{j,k=1}^2
(\partial_k u)\times \partial_{kjj}u,
\label{eq:cross-laplacian-exact}
\\
\Delta\bigl(|\nabla u|^2u\bigr)
&=
|\nabla u|^2\Delta u
+
4\sum_{k,m=1}^2
\bigl(\partial_m u\cdot \partial_{km}u\bigr)\partial_k u
\notag\\
&\quad
+
2\sum_{m=1}^2
\bigl(\partial_m u\cdot \partial_m\Delta u\bigr)u
+
2\sum_{k,m=1}^2
|\partial_{km}u|^2 u.
\label{eq:cubic-laplacian-exact}
\end{align}
\end{lemma}

\begin{proof}
We first prove \eqref{eq:cross-laplacian-exact}. By \eqref{eq:cross-div-exact},
\[
\partial_j\bigl(u\times \partial_j u\bigr)=u\times \partial_{jj}u.
\]
Apply the Laplacian to this identity and, for each $k\in\{1,2\}$, use the second-order product rule
\[
\partial_{kk}(v\times w)
=
(\partial_{kk}v)\times w
+
2(\partial_k v)\times (\partial_k w)
+
v\times \partial_{kk}w,
\]
with $v=u$ and $w=\partial_{jj}u$. We get
\[
\Delta\partial_j(u\times \partial_j u)
=
\sum_{k=1}^2
\Bigl(
(\partial_{kk}u)\times \partial_{jj}u
+
2(\partial_k u)\times \partial_{kjj}u
+
u\times \partial_{kkjj}u
\Bigr).
\]
Now sum over $j=1,2$. The first class of terms satisfies
\[
\sum_{j,k=1}^2(\partial_{kk}u)\times \partial_{jj}u
=
\Bigl(\sum_{k=1}^2\partial_{kk}u\Bigr)\times
\Bigl(\sum_{j=1}^2\partial_{jj}u\Bigr)
=
\Delta u\times \Delta u
=
0,
\]
whereas the third class becomes
\[
\sum_{j,k=1}^2 u\times \partial_{kkjj}u
=
u\times \Delta^2 u.
\]
Thus \eqref{eq:cross-laplacian-exact} follows.

We next prove \eqref{eq:cubic-laplacian-exact}. By \eqref{eq:cubic-diff-exact},
\[
\partial_k\bigl(|\nabla u|^2u\bigr)
=
2\sum_{m=1}^2
\bigl(\partial_m u\cdot \partial_{km}u\bigr)u
+
|\nabla u|^2\partial_k u.
\]
Differentiating once more in $k$ and summing, we obtain
\begin{align*}
\Delta\bigl(|\nabla u|^2u\bigr)
&=
2\sum_{k,m=1}^2
\partial_k\Bigl[\bigl(\partial_m u\cdot \partial_{km}u\bigr)u\Bigr]
+
\sum_{k=1}^2\partial_k\bigl(|\nabla u|^2\partial_k u\bigr)
\\
&=
2\sum_{k,m=1}^2
\partial_k\bigl(\partial_m u\cdot \partial_{km}u\bigr)u
+
2\sum_{k,m=1}^2
\bigl(\partial_m u\cdot \partial_{km}u\bigr)\partial_k u
\\
&\quad
+
\sum_{k=1}^2\partial_k(|\nabla u|^2)\partial_k u
+
|\nabla u|^2\Delta u.
\end{align*}
Now expand the derivatives in the first two terms:
\[
\partial_k\bigl(\partial_m u\cdot \partial_{km}u\bigr)
=
|\partial_{km}u|^2+\partial_m u\cdot \partial_{kkm}u,
\]
and
\[
\partial_k(|\nabla u|^2)
=
2\sum_{m=1}^2\partial_m u\cdot \partial_{km}u.
\]
Substitution and collection of terms give
\begin{align*}
\Delta\bigl(|\nabla u|^2u\bigr)
&=
2\sum_{k,m=1}^2|\partial_{km}u|^2u
+
2\sum_{m=1}^2
\Bigl(\partial_m u\cdot \sum_{k=1}^2\partial_{kkm}u\Bigr)u
\\
&\quad
+
4\sum_{k,m=1}^2
\bigl(\partial_m u\cdot \partial_{km}u\bigr)\partial_k u
+
|\nabla u|^2\Delta u
\\
&=
2\sum_{k,m=1}^2|\partial_{km}u|^2u
+
2\sum_{m=1}^2
\bigl(\partial_m u\cdot \partial_m\Delta u\bigr)u
\\
&\quad
+
4\sum_{k,m=1}^2
\bigl(\partial_m u\cdot \partial_{km}u\bigr)\partial_k u
+
|\nabla u|^2\Delta u,
\end{align*}
which is \eqref{eq:cubic-laplacian-exact}.
\end{proof}

\begin{lemma}[$H_Q^2$ equivalence]\label{lem:H2-equivalence}
Let $Q\in \Sph$ be fixed and let $v:=u-Q\in H^2(\R^2;\R^3)$. Then
\begin{equation}\label{eq:H2-equiv}
\norm{v}_{H^2}^2
\sim
\norm{v}_{L^2_x}^2+\norm{\nabla v}_{L^2_x}^2+\norm{\Delta v}_{L^2_x}^2
=
\norm{u-Q}_{L^2_x}^2+\norm{\nabla u}_{L^2_x}^2+\norm{\Delta u}_{L^2_x}^2.
\end{equation}
The equivalence constants depend only on the dimension.
\end{lemma}

\begin{proof}
Take the Fourier transform of $v$. Since
\[
(1+|\xi|^2)^2 \sim 1+|\xi|^2+|\xi|^4,
\]
we have
\begin{align*}
\norm{v}_{H^2}^2
&=
\int_{\R^2}(1+|\xi|^2)^2 |\widehat v(\xi)|^2\,d\xi
\\
&\sim
\int_{\R^2}\bigl(1+|\xi|^2+|\xi|^4\bigr)|\widehat v(\xi)|^2\,d\xi
\\
&=
\norm{v}_{L^2_x}^2+\norm{\nabla v}_{L^2_x}^2+\norm{\Delta v}_{L^2_x}^2.
\end{align*}
Since $v=u-Q$ and $Q$ is constant, $\nabla v=\nabla u$ and $\Delta v=\Delta u$, which gives \eqref{eq:H2-equiv}.
\end{proof}

We now turn to the automatic low-frequency boundedness on terminal time windows.

The next lemma is the key that links the ultraviolet tail to the full-frequency $L^4$ quantity: on a finite time window, every fixed low-frequency portion below a cutoff is automatically integrable.

\begin{lemma}[Low-frequency budget]\label{lem:low-freq}
Let $u$ be a strong solution of \eqref{eq:LLG-intro}, let $I=(T_0,T_1)\subset[0,\Tstar)$, and let $k\in\mathbb Z$. Then
\begin{align}
\norm{P_{<k}\nabla u}_{L^4_{t,x}(I\times\R^2)}
&\le C|I|^{1/4}2^{k/2}E(u_0)^{1/2},
\label{eq:low-L4-scalar}
\\
\Biggl(\sum_{j<k}\norm{P_j\nabla u}_{L^4_{t,x}(I\times\R^2)}^2\Biggr)^{1/2}
&\le C|I|^{1/4}2^{k/2}E(u_0)^{1/2}.
\label{eq:low-L4-ell2}
\end{align}
In particular, if $I=\Ik_{k_0}(\Tstar;c_0)$ is a self-similar time window, then
\begin{align}
\norm{P_{<k_0+M}\nabla u}_{L^4_{t,x}(\Ik_{k_0}(\Tstar;c_0)\times\R^2)}
&\le C c_0^{1/4}2^{M/2}E(u_0)^{1/2},
\label{eq:low-self-scalar}
\\
\Biggl(\sum_{j<k_0+M}\norm{P_j\nabla u}_{L^4_{t,x}(\Ik_{k_0}(\Tstar;c_0)\times\R^2)}^2\Biggr)^{1/2}
&\le C c_0^{1/4}2^{M/2}E(u_0)^{1/2}.
\label{eq:low-self-ell2}
\end{align}
\end{lemma}

\begin{proof}
We first prove \eqref{eq:low-L4-scalar}. By H\"older's inequality in the time variable,
\[
\norm{P_{<k}\nabla u}_{L^4_{t,x}(I\times\R^2)}
\le |I|^{1/4}\sup_{t\in I}\norm{P_{<k}\nabla u(t)}_{L^4_x}.
\]
The low-frequency Bernstein inequality gives
\[
\norm{P_{<k}\nabla u(t)}_{L^4_x}
\le C2^{k/2}\norm{P_{<k}\nabla u(t)}_{L^2_x}
\le C2^{k/2}\norm{\nabla u(t)}_{L^2_x}.
\]
Combining this with energy monotonicity,
\[
\sup_{t\in I}\norm{\nabla u(t)}_{L^2_x}
\le (2E(u_0))^{1/2},
\]
yields \eqref{eq:low-L4-scalar}.

We next prove \eqref{eq:low-L4-ell2}. For each $j<k$, the same H\"older and Bernstein inequalities give
\[
\norm{P_j\nabla u}_{L^4_{t,x}(I\times\R^2)}^2
\le C|I|^{1/2}2^j\sup_{t\in I}\norm{P_j\nabla u(t)}_{L^2_x}^2.
\]
Therefore
\begin{align*}
\sum_{j<k}\norm{P_j\nabla u}_{L^4_{t,x}(I\times\R^2)}^2
&\le C|I|^{1/2}\sum_{j<k}2^j\sup_{t\in I}\norm{P_j\nabla u(t)}_{L^2_x}^2
\\
&\le C|I|^{1/2}2^k\sup_{t\in I}\sum_{j<k}\norm{P_j\nabla u(t)}_{L^2_x}^2
\\
&\le C|I|^{1/2}2^k E(u_0).
\end{align*}
Taking the square root gives \eqref{eq:low-L4-ell2}.

If $I=\Ik_{k_0}(\Tstar;c_0)$ and $k$ is replaced by $k_0+M$, then
\[
|I|^{1/4}2^{(k_0+M)/2}
\le c_0^{1/4}2^{-k_0/2}2^{k_0/2}2^{M/2}=c_0^{1/4}2^{M/2}.
\]
Substituting this into the preceding estimates yields \eqref{eq:low-self-scalar} and \eqref{eq:low-self-ell2}.
\end{proof}

\begin{remark}
Lemma~\ref{lem:low-freq} says that, on a self-similar time window, the $L^4$ size of the low-frequency part $P_{<k+M}\nabla u$ is completely determined by the energy and the geometric parameters $c_0,M$, independently of the terminal complexity of the solution. Therefore, if $L^4$ control is lost at finite-time blow-up, the only part that can genuinely be responsible is the ultraviolet tail. The frequency decomposition used here is deliberately elementary: only Bernstein inequalities and the energy bound enter.
\end{remark}

\section{\texorpdfstring{$H^2$}{H2} a priori estimates}\label{sec:mainproof}

This section proves the core inequality \eqref{eq:intro-H2-ineq} of the paper. The analytic input needed for it mainly comes from Lemma \ref{lem:derived-GN} and Lemma \ref{lem:pointwise-expansion} in Section \ref{sec:aux}: the former provides the two-dimensional interpolation formulas, and the latter gives the higher-order differential expansions. Together they show that the higher-order energy of LLG is driven only multiplicatively by the critical quantity $\|\nabla u\|_{L^4_x}^4$, which is the fundamental reason why the continuation criteria below hold.

\subsection{\texorpdfstring{$H^2$}{H2} differential inequality}

\begin{proposition}[$H^2$ estimate]\label{prop:H2-main}
Let $u$ be a smooth solution of \eqref{eq:LLG-intro}, and set
\begin{equation}\label{eq:YZ-def}
Y(t):=\norm{\Delta u(t)}_{L^2_x}^2,
\qquad
Z(t):=\norm{\nabla\Delta u(t)}_{L^2_x}^2.
\end{equation}
Then there exists a constant $C_{a,b}>0$, depending only on $a,b$, such that for a.e. $t\in(0,\Tstar)$,
\begin{equation}\label{eq:H2-main}
\frac12\frac{d}{dt}Y(t)+\frac a2 Z(t)
\le C_{a,b}\norm{\nabla u(t)}_{L^4_x}^4 Y(t).
\end{equation}
\end{proposition}

\begin{proof}
Write the divergence form \eqref{eq:LLG-div} as
\[
\partial_t u-a\Delta u
=
b\,\partial_j(u\times \partial_j u)+a|\nabla u|^2u.
\]
Apply $\Delta$ to both sides and take the $L^2$ pairing with $\Delta u$. Since
\begin{align}
\int_{\R^2}\Delta u\cdot \Delta \partial_t u\,dx
&=
\frac12\frac{d}{dt}\norm{\Delta u}_{L^2_x}^2
=
\frac12\frac{d}{dt}Y(t),
\notag
\\
-a\int_{\R^2}\Delta u\cdot \Delta^2u\,dx
&=
a\int_{\R^2}|\nabla\Delta u|^2\,dx
=
aZ(t),
\notag
\end{align}
we obtain
\begin{equation}\label{eq:start-H2}
\frac12\frac{d}{dt}Y(t)+aZ(t)=bI_1(t)+aI_2(t),
\end{equation}
where
\begin{align}
I_1(t)&:=\int_{\R^2}\Delta\partial_j(u\times \partial_j u)\cdot \Delta u\,dx,
\notag
\\
I_2(t)&:=\int_{\R^2}\Delta\bigl(|\nabla u|^2u\bigr)\cdot \Delta u\,dx.
\notag
\end{align}

\medskip
\noindent Step 1: exact simplification of the precession term $I_1$.

By \eqref{eq:cross-laplacian-exact} in Lemma \ref{lem:lap-expansion},
\[
\Delta\partial_j(u\times \partial_j u)
=
u\times \Delta^2 u
+
2\sum_{k=1}^2
(\partial_k u)\times \partial_k\Delta u.
\]
Therefore
\begin{equation}\label{eq:I1-first-split}
I_1(t)=I_{1,1}(t)+2I_{1,2}(t),
\end{equation}
where
\begin{align}
I_{1,1}(t)
&:=
\int_{\R^2}
\bigl(u\times \Delta^2 u\bigr)\cdot \Delta u\,dx,
\notag
\\
I_{1,2}(t)
&:=
\sum_{k=1}^2
\int_{\R^2}
\bigl((\partial_k u)\times \partial_k\Delta u\bigr)\cdot \Delta u\,dx.
\label{eq:I12-def}
\end{align}
For $I_{1,1}$, first use the scalar triple product identity
\[
(a\times b)\cdot c =-(a\times c)\cdot b
\]
to rewrite the integrand as
\[
\bigl(u\times \Delta^2u\bigr)\cdot \Delta u
=
-\bigl(u\times \Delta u\bigr)\cdot \Delta^2u.
\]
Then integrate by parts in $x_k$ to get
\begin{align}
I_{1,1}(t)
&=
-\int_{\R^2}
\bigl(u\times \Delta u\bigr)\cdot \Delta^2u\,dx
\notag\\
&=
\sum_{k=1}^2
\int_{\R^2}
\partial_k\bigl(u\times \Delta u\bigr)\cdot \partial_k\Delta u\,dx
\notag\\
&=
\sum_{k=1}^2
\int_{\R^2}
\bigl((\partial_k u)\times \Delta u\bigr)\cdot \partial_k\Delta u\,dx
+
\sum_{k=1}^2
\int_{\R^2}
\bigl(u\times \partial_k\Delta u\bigr)\cdot \partial_k\Delta u\,dx
\notag\\
&=
\sum_{k=1}^2
\int_{\R^2}
\bigl((\partial_k u)\times \Delta u\bigr)\cdot \partial_k\Delta u\,dx,
\label{eq:I11-after-parts}
\end{align}
because
\[
\bigl(u\times \partial_k\Delta u\bigr)\cdot \partial_k\Delta u=0.
\]
On the other hand, the same scalar triple product identity rewrites \eqref{eq:I12-def} as
\begin{equation}\label{eq:I12-triple}
I_{1,2}(t)
=
-\sum_{k=1}^2
\int_{\R^2}
\bigl((\partial_k u)\times \Delta u\bigr)\cdot \partial_k\Delta u\,dx.
\end{equation}
Substituting \eqref{eq:I11-after-parts} and \eqref{eq:I12-triple} back into \eqref{eq:I1-first-split} gives the exact identity
\begin{equation*}
I_1(t)
=
-\sum_{k=1}^2
\int_{\R^2}
\bigl((\partial_k u)\times \Delta u\bigr)\cdot \partial_k\Delta u\,dx.
\end{equation*}
Equivalently, this step may be written as the following pointwise divergence identity:
\begin{equation*}
\Bigl[\sum_{j=1}^2\Delta\partial_j\bigl(u\times \partial_j u\bigr)\Bigr]\cdot \Delta u
+
\sum_{k=1}^2
\bigl((\partial_k u)\times \Delta u\bigr)\cdot \partial_k\Delta u
=
\sum_{k=1}^2
\partial_k\Bigl[\bigl(u\times \partial_k\Delta u\bigr)\cdot \Delta u\Bigr].
\end{equation*}
Thus
\begin{equation*}
|I_1(t)|
\le
C\int_{\R^2}
|\nabla u|\,|\Delta u|\,|\nabla\Delta u|\,dx.
\end{equation*}
By H\"older's inequality and \eqref{eq:Delta-L4},
\begin{align}
|I_1(t)|
&\le
C\norm{\nabla u(t)}_{L^4_x}
\norm{\Delta u(t)}_{L^4_x}
\norm{\nabla\Delta u(t)}_{L^2_x}
\notag\\
&\le
C\norm{\nabla u(t)}_{L^4_x}
\bigl(Y(t)^{1/4}Z(t)^{1/4}\bigr)
Z(t)^{1/2}
\notag\\
&=
C\norm{\nabla u(t)}_{L^4_x}Y(t)^{1/4}Z(t)^{3/4}.
\notag
\end{align}
We apply Young's inequality to the last term. Let
\[
A:=Z(t)^{3/4},
\qquad
B:=\norm{\nabla u(t)}_{L^4_x}Y(t)^{1/4}.
\]
Then $AB\le \varepsilon A^{4/3}+C_\varepsilon B^4$, namely
\[
\norm{\nabla u}_{L^4_x}Y^{1/4}Z^{3/4}
\le
\varepsilon Z
+
C_\varepsilon \norm{\nabla u}_{L^4_x}^4Y.
\]
Choosing
\[
\varepsilon=\frac{a}{8(1+|b|)},
\]
and absorbing constants into $C_{a,b}$, we obtain
\begin{equation}\label{eq:I1-final}
|bI_1(t)|
\le
\frac a8 Z(t)
+
C_{a,b}\norm{\nabla u(t)}_{L^4_x}^4Y(t).
\end{equation}

\medskip
\noindent Step 2: term-by-term expansion of the damping nonlinearity $I_2$.

By \eqref{eq:cubic-laplacian-exact} in Lemma \ref{lem:lap-expansion},
\[
I_2(t)=I_{2,1}(t)+I_{2,2}(t)+I_{2,3}(t)+I_{2,4}(t),
\]
where
\begin{align}
I_{2,1}(t)
&:=
\int_{\R^2}
|\nabla u|^2|\Delta u|^2\,dx,
\notag
\\
I_{2,2}(t)
&:=
4\sum_{k,m=1}^2
\int_{\R^2}
\bigl(\partial_m u\cdot \partial_{km}u\bigr)
\bigl(\partial_k u\cdot \Delta u\bigr)\,dx,
\notag
\\
I_{2,3}(t)
&:=
2\sum_{m=1}^2
\int_{\R^2}
\bigl(\partial_m u\cdot \partial_m\Delta u\bigr)
\bigl(u\cdot \Delta u\bigr)\,dx,
\notag
\\
I_{2,4}(t)
&:=
2\sum_{k,m=1}^2
\int_{\R^2}
|\partial_{km}u|^2
\bigl(u\cdot \Delta u\bigr)\,dx.
\notag
\end{align}
By the constraint identity \eqref{eq:uDelta},
\begin{equation}\label{eq:I24-sign}
I_{2,4}(t)
=
-2\sum_{k,m=1}^2
\int_{\R^2}
|\partial_{km}u|^2|\nabla u|^2\,dx
\le 0.
\end{equation}

For $I_{2,1}$, H\"older's inequality and \eqref{eq:Delta-L4} give
\begin{align}
|I_{2,1}(t)|
&\le
\norm{\nabla u(t)}_{L^4_x}^2
\norm{\Delta u(t)}_{L^4_x}^2
\notag\\
&\le
C\norm{\nabla u(t)}_{L^4_x}^2
Y(t)^{1/2}Z(t)^{1/2}
\notag\\
&\le
\frac18 Z(t)+C\norm{\nabla u(t)}_{L^4_x}^4Y(t).
\label{eq:I21-bound}
\end{align}

For $I_{2,2}$, using H\"older's inequality, \eqref{eq:nabla2-L4}, and \eqref{eq:Delta-L4}, we get
\begin{align}
|I_{2,2}(t)|
&\le
C
\norm{\nabla u(t)}_{L^4_x}^2
\norm{\nabla^2u(t)}_{L^4_x}
\norm{\Delta u(t)}_{L^4_x}
\notag\\
&\le
C
\norm{\nabla u(t)}_{L^4_x}^2
\bigl(Y(t)^{1/4}Z(t)^{1/4}\bigr)
\bigl(Y(t)^{1/4}Z(t)^{1/4}\bigr)
\notag\\
&=
C\norm{\nabla u(t)}_{L^4_x}^2
Y(t)^{1/2}Z(t)^{1/2}
\notag\\
&\le
\frac18 Z(t)+C\norm{\nabla u(t)}_{L^4_x}^4Y(t).
\label{eq:I22-bound}
\end{align}

For $I_{2,3}$, first use \eqref{eq:uDelta} to write
\[
I_{2,3}(t)
=
-2\sum_{m=1}^2
\int_{\R^2}
\bigl(\partial_m u\cdot \partial_m\Delta u\bigr)
|\nabla u|^2\,dx.
\]
Hence, by H\"older's inequality and \eqref{eq:nabla-L6},
\begin{align}
|I_{2,3}(t)|
&\le
C
\norm{\nabla\Delta u(t)}_{L^2_x}
\norm{\nabla u(t)}_{L^6_x}^3
\notag\\
&\le
C
Z(t)^{1/2}
\norm{\nabla u(t)}_{L^4_x}^2
Y(t)^{1/2}
\notag\\
&\le
\frac18 Z(t)+C\norm{\nabla u(t)}_{L^4_x}^4Y(t).
\label{eq:I23-bound}
\end{align}

The fourth term can be discarded directly by \eqref{eq:I24-sign}. Combining \eqref{eq:I21-bound}, \eqref{eq:I22-bound}, \eqref{eq:I23-bound}, and \eqref{eq:I24-sign}, we find
\begin{equation}\label{eq:I2-final}
aI_2(t)
\le
\frac{3a}{8}Z(t)
+
C_{a,b}\norm{\nabla u(t)}_{L^4_x}^4Y(t).
\end{equation}

\medskip
\noindent Step 3: absorption of the dissipation term.

Substituting \eqref{eq:I1-final} and \eqref{eq:I2-final} into \eqref{eq:start-H2}, we obtain
\[
\frac12\frac{d}{dt}Y(t)+aZ(t)
\le
\frac a8 Z(t)+\frac{3a}{8}Z(t)
+
C_{a,b}\norm{\nabla u(t)}_{L^4_x}^4Y(t).
\]
Moving the $Z(t)$ terms on the right-hand side to the left gives
\[
\frac12\frac{d}{dt}Y(t)+\frac a2 Z(t)
\le C_{a,b}\norm{\nabla u(t)}_{L^4_x}^4Y(t),
\]
which is \eqref{eq:H2-main}.
\end{proof}

\begin{lemma}[Smoothing transfer]\label{lem:strong-H2-transfer}
Let $u$ be a strong solution of \eqref{eq:LLG-intro} on $[0,\Tstar)$, and let $0\le T_0<T_1<\Tstar$. Then for every $\varepsilon\in(0,T_1-T_0)$, $u$ is smooth on the strip
\[
[T_0+\varepsilon,T_1]\times\R^2.
\]
In particular, if $Y,Z$ are defined by \eqref{eq:YZ-def}, then for every $t\in[T_0,T_1]$,
\begin{equation}\label{eq:strong-H2-integral}
Y(t)+a\int_{T_0}^{t} Z(s)\,ds
\le
Y(T_0)+C_{a,b}\int_{T_0}^{t}\norm{\nabla u(s)}_{L^4_x}^4Y(s)\,ds.
\end{equation}
\end{lemma}

\begin{proof}
By standard quasilinear parabolic regularization, for every fixed $\varepsilon>0$, the strong solution $u$ is smooth on
\[
[T_0+\varepsilon,T_1]\times\R^2.
\]
Therefore Proposition \ref{prop:H2-main} may be applied pointwise on this strip. For every $t\in[T_0+\varepsilon,T_1]$, integrating \eqref{eq:H2-main} over $[T_0+\varepsilon,t]$ yields
\[
Y(t)+a\int_{T_0+\varepsilon}^{t} Z(s)\,ds
\le
Y(T_0+\varepsilon)+C_{a,b}\int_{T_0+\varepsilon}^{t}\norm{\nabla u(s)}_{L^4_x}^4Y(s)\,ds.
\]
By the strong-solution continuity $u-Q\in C([T_0,T_1];H^2)$,
\[
Y(T_0+\varepsilon)\to Y(T_0)\qquad(\varepsilon\downarrow0).
\]
On the other hand, by Lemma \ref{lem:L4-by-Lap},
\[
\norm{\nabla u(s)}_{L^4_x}^4Y(s)
\le Y(s)^2,
\]
and $Y$ is continuous on $[T_0,T_1]$, so the integrand on the right belongs to $L^1(T_0,T_1)$. Letting $\varepsilon\downarrow0$ and using Fatou's lemma for the nonnegative integral term on the left gives \eqref{eq:strong-H2-integral}.
\end{proof}

\subsection{Gronwall consequences}

\begin{corollary}[Gronwall bound]\label{cor:Gronwall}
Let $u$ be a maximal strong solution of \eqref{eq:LLG-intro}. If on some interval $I=(T_0,T_1)\subset[0,\Tstar)$,
\begin{equation}\label{eq:M-def}
M_I^4:=\int_{T_0}^{T_1}\norm{\nabla u(t)}_{L^4_x}^4\,dt<\infty,
\end{equation}
then
\begin{equation}\label{eq:H2-bound-I}
\sup_{t\in[T_0,T_1]}\norm{\Delta u(t)}_{L^2_x}^2
\le \norm{\Delta u(T_0)}_{L^2_x}^2\exp\bigl(C_{a,b}M_I^4\bigr),
\end{equation}
and
\begin{equation}\label{eq:H3-time-int}
\int_{T_0}^{T_1}\norm{\nabla\Delta u(t)}_{L^2_x}^2\,dt
\le C_{a,b}\norm{\Delta u(T_0)}_{L^2_x}^2\exp\bigl(C_{a,b}M_I^4\bigr).
\end{equation}
Consequently, together with Lemma~\ref{lem:L2-relative}, one has
\begin{equation*}
\sup_{t\in[T_0,T_1]}\norm{u(t)-Q}_{H^2}
\le C\Bigl(
\norm{u(T_0)-Q}_{H^2},\,E(u_0),\,a,\,b,\,M_I,\,T_1-T_0\Bigr).
\end{equation*}
\end{corollary}

\begin{proof}
By Lemma \ref{lem:strong-H2-transfer}, for every $t\in[T_0,T_1]$,
\[
Y(t)
\le
Y(T_0)+C_{a,b}\int_{T_0}^{t}\norm{\nabla u(s)}_{L^4_x}^4Y(s)\,ds.
\]
Thus Gronwall's inequality gives
\[
Y(t)\le Y(T_0)\exp\!\Bigl(C_{a,b}\int_{T_0}^{t}\norm{\nabla u(s)}_{L^4_x}^4\,ds\Bigr)
\le Y(T_0)\exp\!\bigl(C_{a,b}M_I^4\bigr),
\]
which proves \eqref{eq:H2-bound-I}.

Substituting this uniform bound back into \eqref{eq:strong-H2-integral} and taking $t=T_1$, we obtain
\[
a\int_{T_0}^{T_1}Z(t)\,dt
\le
Y(T_0)+C_{a,b}\sup_{t\in[T_0,T_1]}Y(t)
\int_{T_0}^{T_1}\norm{\nabla u(t)}_{L^4_x}^4\,dt.
\]
By \eqref{eq:H2-bound-I} and \eqref{eq:M-def},
\[
\int_{T_0}^{T_1}Z(t)\,dt
\le
C_{a,b}Y(T_0)e^{C_{a,b}M_I^4},
\]
which is \eqref{eq:H3-time-int}.

Finally, the integrated form of Lemma~\ref{lem:L2-relative} on $[T_0,T_1]$ gives
\[
\norm{u(t)-Q}_{L^2_x}^2
\le
\norm{u(T_0)-Q}_{L^2_x}^2+4aE(u_0)(t-T_0)
\le
\norm{u(T_0)-Q}_{L^2_x}^2+4aE(u_0)(T_1-T_0).
\]
Together with energy monotonicity,
\[
\norm{\nabla u(t)}_{L^2_x}^2\le 2E(u_0),
\]
and Lemma \ref{lem:H2-equivalence},
\[
\norm{u(t)-Q}_{H^2}
\le
C\Bigl(
\norm{u(t)-Q}_{L^2_x}+\norm{\nabla u(t)}_{L^2_x}+\norm{\Delta u(t)}_{L^2_x}
\Bigr),
\]
we get
\[
\sup_{t\in[T_0,T_1]}\norm{u(t)-Q}_{H^2}
\le C\Bigl(
\norm{u(T_0)-Q}_{H^2},\,E(u_0),\,a,\,b,\,M_I,\,T_1-T_0
\Bigr).
\]
\end{proof}

\begin{remark}
Corollary \ref{cor:Gronwall} is the fundamental reason for choosing $L^4_{t,x}$ as the controlling norm in this paper. On the one hand, the $L^4_{t,x}$ norm of $\nabla u$ is invariant under the two-dimensional critical scaling. On the other hand, Proposition \ref{prop:H2-main} writes the higher-order energy exactly in a Gronwall-type structure of the form ``critical quantity times higher-order quantity.'' This combination matches the scale of two-dimensional energy-critical geometric flows.
\end{remark}

\section{Continuation criteria and subthreshold global existence}\label{sec:proofs}

This section proves the continuation criteria and blow-up consequences stated in Section 1 one by one. Section \ref{sec:mainproof} has already provided all the genuine a priori estimates; from this point on, we only need to splice these a priori estimates with the standard local strong-solution theory.

In the proofs below, we repeatedly use the following quantitative consequence of the local theory: for every $R>0$, there exist
\[
\delta_{\mathrm{loc}}(R)>0,
\qquad
C_{\mathrm{loc}}(R)>0,
\]
such that whenever at some time $t_0$ one has
\[
\|u(t_0)-Q\|_{H^2}\le R,
\]
the strong solution with initial value $u(t_0)$ exists at least on the interval $[t_0,t_0+\delta_{\mathrm{loc}}(R)]$ and satisfies
\[
\sup_{t\in[t_0,t_0+\delta_{\mathrm{loc}}(R)]}\|u(t)-Q\|_{H^2}\le C_{\mathrm{loc}}(R).
\]
This is a standard consequence of the local theory for quasilinear parabolic equations; see \cite{CarbouFabrie,Melcher,GuoDingBook,WangGuo}.

\subsection{Finite-terminal full-frequency \texorpdfstring{$L^4$}{L4} continuation criterion}

This subsection proves Theorem \ref{thm:L4-cont}. Section \ref{sec:mainproof} has already supplied all the genuine a priori estimates; what remains here is only to splice these a priori estimates with the standard local strong-solution theory.

In the proof below, we repeatedly use the following quantitative consequence of the local theory: for every $R>0$, there exist
\[
\delta_{\mathrm{loc}}(R)>0,
\qquad
C_{\mathrm{loc}}(R)>0,
\]
such that whenever at some time $t_0$ one has
\[
\|u(t_0)-Q\|_{H^2}\le R,
\]
the strong solution with initial value $u(t_0)$ exists at least on the interval $[t_0,t_0+\delta_{\mathrm{loc}}(R)]$ and satisfies
\[
\sup_{t\in[t_0,t_0+\delta_{\mathrm{loc}}(R)]}\|u(t)-Q\|_{H^2}\le C_{\mathrm{loc}}(R).
\]
This is a standard consequence of the local theory for quasilinear parabolic equations; see \cite{CarbouFabrie,Melcher,GuoDingBook,WangGuo}.

\begin{proof}[Proof of Theorem \ref{thm:L4-cont}]
Assume that $0<\Tstar<\infty$ and \eqref{eq:L4-ass} holds. Set
\[
M_\ast^4:=\int_{T_0}^{\Tstar}\norm{\nabla u(t)}_{L^4_x}^4\,dt<\infty.
\]
For any $t\in(T_0,\Tstar)$, apply Corollary \ref{cor:Gronwall} to the interval $I=(T_0,t)$ and obtain
\begin{equation*}
\sup_{T_0\le s\le t}\norm{u(s)-Q}_{H^2}
\le
C\Bigl(
\norm{u(T_0)-Q}_{H^2},\,E(u_0),\,a,\,b,\,M_\ast,\,\Tstar-T_0
\Bigr).
\end{equation*}
The right-hand side is independent of $t$, and therefore there exists a constant $R_\ast<\infty$ such that
\begin{equation}\label{eq:global-H2-bound-before-T}
\sup_{T_0\le s<\Tstar}\norm{u(s)-Q}_{H^2}\le R_\ast.
\end{equation}

We now invoke the quantitative local theory written above. Put $R=R_\ast$ and obtain a positive number depending only on $R_\ast$:
\[
\delta_\ast:=\delta_{\mathrm{loc}}(R_\ast)>0.
\]
Take any sequence $t_n\uparrow \Tstar$, and choose $n$ so large that
\[
\Tstar-t_n<\frac{\delta_\ast}{2}.
\]
By \eqref{eq:global-H2-bound-before-T},
\[
\|u(t_n)-Q\|_{H^2}\le R_\ast.
\]
Therefore the strong solution with initial value $u(t_n)$ exists at least on
\[
[t_n,\,t_n+\delta_\ast].
\]
By uniqueness of strong solutions, this new solution must coincide with the original maximal solution on the overlapping interval $[t_n,\Tstar)$. Hence the original solution can be continued to
\[
[t_n,\,t_n+\delta_\ast]\supset [t_n,\,\Tstar+\delta_\ast/2].
\]
This says that the original maximal solution actually extends beyond $\Tstar$, contradicting the maximality of $\Tstar$. Thus the solution can be continued beyond $t=\Tstar$.
\end{proof}

\subsection{Single ultraviolet-tail continuation criterion}

This subsection proves Theorem \ref{thm:uv-tail-cont}. The key point is the low-frequency budget from Section~\ref{sec:prelim}: once a fixed high-frequency tail belongs to $L^4_{t,x}$ on the terminal interval, the remaining low-frequency part is automatically supplied by the energy and the interval length.

\begin{proof}[Proof of Theorem \ref{thm:uv-tail-cont}]
Let $I=(T_0,\Tstar)$ and assume that \eqref{eq:uv-tail-ass} holds. Decompose the gradient as
\[
\nabla u=P_{<k_\ast}\nabla u+P_{\ge k_\ast}\nabla u.
\]
By the triangle inequality,
\begin{equation}\label{eq:uv-tail-split}
\norm{\nabla u}_{L^4_{t,x}(I\times\R^2)}
\le
\norm{P_{<k_\ast}\nabla u}_{L^4_{t,x}(I\times\R^2)}
+
\norm{P_{\ge k_\ast}\nabla u}_{L^4_{t,x}(I\times\R^2)}.
\end{equation}
The second term is exactly \eqref{eq:uv-tail-ass}. For the first term, \eqref{eq:low-L4-scalar} in Lemma~\ref{lem:low-freq} gives
\begin{equation}\label{eq:uv-low-budget}
\norm{P_{<k_\ast}\nabla u}_{L^4_{t,x}(I\times\R^2)}
\le C(\Tstar-T_0)^{1/4}2^{k_\ast/2}E(u_0)^{1/2}<\infty.
\end{equation}
Putting \eqref{eq:uv-tail-ass} and \eqref{eq:uv-low-budget} into \eqref{eq:uv-tail-split}, we obtain
\[
\norm{\nabla u}_{L^4_{t,x}(I\times\R^2)}<\infty.
\]
Theorem \ref{thm:L4-cont} then implies that the solution can be continued beyond $t=\Tstar$. This proves Theorem \ref{thm:uv-tail-cont}.
\end{proof}

\subsection{Global strong solutions derived from subthreshold energy}

All results in the previous subsection concern only continuation at a \emph{finite terminal time}; hence they require neither the $4\pi$ threshold nor any smallness assumption. This subsection explains that, once the finite-terminal criteria of this paper are combined with the existing subthreshold global theory, one immediately obtains a genuine global strong-solution result.

To avoid continuing to use the specialized notation of the external literature, we agree that, for any time interval $I\subset[0,\infty)$,
\begin{equation}\label{eq:def-weak-class}
\begin{aligned}
\mathcal W(I\times\R^2):=\bigl\{U:I\times\R^2\to\Sph:
&\ U-Q\in L^{\infty}(I;L^2(\R^2;\R^3)),
\\
&\ \nabla U\in L^{\infty}(I;L^2(\R^2)),
\\
&\ \partial_tU\in L^2(I\times\R^2),
\\
&\ U\text{ satisfies }\eqref{eq:LLG-intro}\text{ in }\mathcal D'(I\times\R^2),
\\
&\ E(U(t_2))+a\int_{t_1}^{t_2}\!\!\int_{\R^2}|\partial_tU|^2\,dxdt\le E(U(t_1)),
\\
&\ \text{for a.e. }0\le t_1\le t_2\in I\bigr\}.
\end{aligned}
\end{equation}
In what follows, we only need two outputs from \cite[Proposition~2.1, Lemma~2.5, Corollary~1.1]{LiZhao}: the global unique $\mathcal W$-solution corresponding to subthreshold data, and the associated global $L^4_{t,x}$ control.

\begin{theorem}[Subthreshold input]\label{thm:weak-global-input}
Let $u_0\in W^{1,2}(\R^2;\Sph)$ and assume $E(u_0)<4\pi$. Then equation \eqref{eq:LLG-intro} has a unique global solution
\[
U\in \mathcal W([0,\infty)\times\R^2),
\]
and it satisfies
\begin{equation}\label{eq:weak-global-L4}
\int_0^{\infty}\int_{\R^2}|\nabla U(t,x)|^4\,dxdt<\infty.
\end{equation}
Moreover, $E(U(t))\to 0$ as $t\to\infty$.
\end{theorem}

\begin{proof}
This is precisely the conclusion of \cite[Proposition~2.1, Lemma~2.5, Corollary~1.1]{LiZhao} when the target manifold is $N=\Sph$.
\end{proof}

\begin{lemma}[Strong-to-weak class]\label{lem:strong-in-weakclass}
Let $u$ be a strong solution of \eqref{eq:LLG-intro} on $[0,T]$, namely
\[
u-Q\in C([0,T];H^2(\R^2;\R^3))\cap L^2((0,T);H^3(\R^2;\R^3)),
\qquad
\partial_tu\in L^2((0,T);H^1(\R^2;\R^3)).
\]
Then
\[
u\in \mathcal W([0,T]\times\R^2).
\]
\end{lemma}

\begin{proof}
The strong-solution regularity immediately gives
\begin{equation*}
 u-Q\in C([0,T];L^2(\R^2;\R^3)),
 \qquad
 \nabla u\in L^{\infty}(0,T;L^2(\R^2)),
 \qquad
 \partial_tu\in L^2((0,T)\times\R^2).
\end{equation*}
Thus $u$ satisfies the first three spatial conditions in \eqref{eq:def-weak-class}.

Let $\eta\in C_c^{\infty}((0,T)\times\R^2;\R^3)$ be arbitrary. Using
\[
\partial_j(u\times \partial_j u)=u\times \partial_{jj}u
\]
and equation \eqref{eq:LLG-intro}, we have
\begin{align*}
0
&=\iint_{(0,T)\times\R^2}
\bigl(u_t-a(\Delta u+|\nabla u|^2u)-b\,u\times\Delta u\bigr)\cdot \eta\,dxdt
\\
&=-\iint u\cdot \partial_t\eta\,dxdt
+a\iint \nabla u: \nabla \eta\,dxdt
-a\iint |\nabla u|^2u\cdot \eta\,dxdt
\\
&\qquad
-b\sum_{j=1}^2\iint (u\times \partial_j u)\cdot \partial_j\eta\,dxdt.
\end{align*}
Moreover, by the two-dimensional Sobolev embedding $H^2(\R^2)\hookrightarrow W^{1,4}(\R^2)$,
\begin{equation*}
\nabla u\in L^{\infty}(0,T;L^4(\R^2)),
\qquad
\Delta u\in L^{\infty}(0,T;L^2(\R^2)),
\end{equation*}
and hence
\[
|\nabla u|^2u\in L^{\infty}(0,T;L^2(\R^2)),
\qquad
u\times \Delta u\in L^{\infty}(0,T;L^2(\R^2)).
\]
Therefore all terms above are meaningful, and $u$ satisfies \eqref{eq:LLG-intro} in $\mathcal D'((0,T)\times\R^2)$.

Finally, by Lemma~\ref{lem:energy}, for every $0\le t_1\le t_2\le T$,
\begin{equation*}
E(u(t_2))+a\int_{t_1}^{t_2}\int_{\R^2}|\partial_tu|^2\,dxdt
=E(u(t_1)).
\end{equation*}
In particular, $u$ satisfies the energy inequality in \eqref{eq:def-weak-class}. Hence
\[
u\in \mathcal W([0,T]\times\R^2).
\]
\end{proof}

\begin{proof}[Proof of Theorem \ref{thm:below4pi-global}]
Let $u$ be the maximal strong solution with initial data $u_0$. Suppose, to the contrary, that $\Tstar<\infty$. Then for every $T<\Tstar$, Lemma \ref{lem:strong-in-weakclass} gives
\[
u\in \mathcal W([0,T]\times\R^2).
\]
On the other hand, by Theorem \ref{thm:weak-global-input}, there exists a unique global solution
\[
U\in \mathcal W([0,\infty)\times\R^2)
\]
corresponding to the same initial datum $u_0$. By uniqueness in $\mathcal W$, necessarily
\[
u(t,x)=U(t,x),
\qquad
0\le t\le T.
\]
Since $T<\Tstar$ is arbitrary,
\begin{equation}\label{eq:strong-equals-weak}
u(t,x)=U(t,x)
\qquad\text{for a.e. }(t,x)\in[0,\Tstar)\times\R^2.
\end{equation}
Thus \eqref{eq:weak-global-L4} immediately implies that, for every $T_0\in[0,\Tstar)$,
\[
\int_{T_0}^{\Tstar}\int_{\R^2}|\nabla u(t,x)|^4\,dxdt
\le
\int_0^{\infty}\int_{\R^2}|\nabla U(t,x)|^4\,dxdt
<\infty.
\]
Since under the current assumption $0<\Tstar<\infty$, Theorem \ref{thm:L4-cont} applies directly and implies that $u$ can be continued beyond $t=\Tstar$. This contradicts the maximality of $\Tstar$. Therefore
\[
\Tstar=\infty.
\]
For each finite $T>0$, the strong-solution regularity
\[
u-Q\in C([0,T];H^2)\cap L^2((0,T);H^3),
\qquad
\partial_tu\in L^2((0,T);H^1)
\]
is exactly part of the definition of the maximal strong solution. Finally, \eqref{eq:strong-equals-weak} and \eqref{eq:weak-global-L4} give
\[
\int_0^{\infty}\int_{\R^2}|\nabla u(t,x)|^4\,dxdt<\infty.
\]
This proves Theorem \ref{thm:below4pi-global}.
\end{proof}

\begin{remark}
The proof above clearly separates ``finite-terminal continuation'' from ``subthreshold global existence.'' What is established internally in this paper is the finite-terminal criterion represented by Theorem \ref{thm:L4-cont} and Theorem \ref{thm:uv-tail-cont}; the $4\pi$ threshold is used only after the external subthreshold global theory of \cite{LiZhao} is introduced, in order to exclude finite terminal times. Thus, from a logical point of view, $4\pi$ is not a necessary premise of Proposition \ref{prop:H2-main} or the continuation argument following it; it is only an additional input used to upgrade the continuation theorem into a global theorem.
\end{remark}

\section{Degree-zero \texorpdfstring{$8\pi$}{8pi} obstruction}\label{sec:deg0}

The first four sections established an analytic continuation mechanism: they characterize which $L^4$ or ultraviolet-tail quantities must diverge at a finite terminal time, but they do not distinguish different topological sectors. This section is only a supplementary application. We add the degree-zero constraint defined in the introduction by \eqref{eq:degree-def}, namely
\[
\operatorname{deg}(u_0)=0,
\]
and invoke the classical Struwe-type bubble-profile decomposition. Our goal is to prove Theorem~\ref{thm:deg0-8pi-global} from the introduction: in this special sector, the minimal energy cost of a finite-time bubble tree is at least two copies of $4\pi$, and hence the general $4\pi$ energy obstruction can be pushed further up to $8\pi$. This improvement does not come from a new $H^2$ differential inequality, but from the combination of ``the strong-solution framework of the first four sections plus classical bubbling input plus degree-zero quantization.''

The topological degree has already been defined in \eqref{eq:degree-def}. We first prove two basic properties of this notation at the $H_Q^1$ level: the topological energy lower bound and conservation along strong-solution trajectories.

\begin{lemma}[Degree and energy]\label{lem:degree-energy}
Let $v\in H_Q^1(\R^2;\Sph)$. Then
\begin{equation}\label{eq:degree-energy}
4\pi\,\bigl|\operatorname{deg}(v)\bigr|\le E(v).
\end{equation}
If in addition
\[
v\in C\bigl([0,T];H_Q^2(\R^2;\R^3)\bigr),
\]
then $t\mapsto \operatorname{deg}(v(t))$ is constant.
\end{lemma}

\begin{proof}
We first prove \eqref{eq:degree-energy}. By $|v|=1$, $|a\times b|\le |a|\,|b|$, and the arithmetic--geometric mean inequality
\[
ab\le \frac{a^2+b^2}{2}\qquad (a,b\ge 0),
\]
we have pointwise
\[
\bigl|v\cdot(\partial_1 v\times \partial_2 v)\bigr|
\le |v|\,|\partial_1 v\times \partial_2 v|
\le |\partial_1 v|\,|\partial_2 v|
\le \frac12\bigl(|\partial_1 v|^2+|\partial_2 v|^2\bigr).
\]
Integrating over $\R^2$ yields
\[
4\pi\,\bigl|\operatorname{deg}(v)\bigr|
\le
\frac12\int_{\R^2}|\nabla v|^2\,dx
=E(v),
\]
which is \eqref{eq:degree-energy}.

We next explain why the degree is preserved along an $H_Q^2$ continuous path. The first part has already shown that, whenever $v\in H_Q^1$, the degree is meaningful and satisfies \eqref{eq:degree-energy}. To prove conservation along a time-continuous trajectory, we further use the two-dimensional embedding $H^2(\R^2)\hookrightarrow L^\infty(\R^2)$ to control the zeroth-order difference terms in $\mathcal J(v)-\mathcal J(w)$. Set
\[
\mathcal J(w):=w\cdot(\partial_1 w\times \partial_2 w).
\]
For any $v,w\in H_Q^2(\R^2;\Sph)$, the identity
\begin{align*}
\mathcal J(v)-\mathcal J(w)
={}&
(v-w)\cdot(\partial_1 v\times \partial_2 v)
+w\cdot\bigl((\partial_1 v-\partial_1 w)\times \partial_2 v\bigr)
\\
&
+w\cdot\bigl(\partial_1 w\times (\partial_2 v-\partial_2 w)\bigr)
\end{align*}
and H\"older's inequality give
\begin{align*}
\norm{\mathcal J(v)-\mathcal J(w)}_{L^1_x}
\le{}&
\norm{v-w}_{L^\infty_x}\norm{\nabla v}_{L^2_x}^2
\\
&+\norm{\nabla(v-w)}_{L^2_x}\norm{\nabla v}_{L^2_x}
+\norm{\nabla w}_{L^2_x}\norm{\nabla(v-w)}_{L^2_x}.
\end{align*}
By the two-dimensional Sobolev embedding $H^2(\R^2)\hookrightarrow L^\infty(\R^2)$, the right-hand side is controlled by $\norm{v-w}_{H^2_x}$. Therefore, if $v_n\to v$ in $H_Q^2$, then $\operatorname{deg}(v_n)\to \operatorname{deg}(v)$. Hence, if $v\in C([0,T];H_Q^2)$, the map $t\mapsto \operatorname{deg}(v(t))$ is continuous; since the degree is integer-valued, this function must be constant.
\end{proof}

We next keep only one classical bubbling input, stated in the minimal form actually needed later. We say that
\[
U\in \mathcal W([0,T_{\mathrm s})\times\R^2)\cap C^\infty((0,T_{\mathrm s})\times\R^2;\Sph)
\]
has a \emph{finite singularity} at $T_{\mathrm s}$ if there do not exist $\delta>0$ and
\[
\widetilde U\in \mathcal W([0,T_{\mathrm s}+\delta)\times\R^2)\cap C^\infty((0,T_{\mathrm s}+\delta)\times\R^2;\Sph)
\]
such that $\widetilde U=U$ on $(0,T_{\mathrm s})\times\R^2$.

\begin{theorem}[Bubble decomposition]\label{thm:bubble-degree-input}
Let
\[
U\in \mathcal W([0,T_{\mathrm s})\times \R^2)\cap C^\infty((0,T_{\mathrm s})\times\R^2;\Sph),
\qquad
0<T_{\mathrm s}<\infty,
\]
and assume that $T_{\mathrm s}$ is a finite singularity time in the sense above. Then one can choose a subsequence $t_n\uparrow T_{\mathrm s}$, and there exist
\[
u_\ast\in H_Q^1(\R^2;\Sph),
\qquad
J\in \mathbb N,
\]
finitely many singular points $a_1,\dots,a_N\in \R^2$, finitely many scales and centers
\[
x_{j,n}\in \R^2,
\qquad
\lambda_{j,n}>0,
\qquad
\lambda_{j,n}\downarrow0
\quad (1\le j\le J),
\]
finitely many nonconstant harmonic maps
\[
\omega_j:\mathbb S^2\to \mathbb S^2
\qquad (1\le j\le J),
\]
and an index map $\ell:\{1,\dots,J\}\to\{1,\dots,N\}$, such that the following conclusions hold. There exists $\rho_0>0$ such that the closed balls $\overline{B(a_\alpha,2\rho_0)}$ are pairwise disjoint, and for every $0<\rho<\rho_0$, if
\[
K_\rho:=\R^2\setminus \bigcup_{\alpha=1}^N B(a_\alpha,\rho),
\]
then
\begin{align}
U(t_n)-Q &\rightharpoonup u_\ast-Q
\qquad\text{weakly in }H^1(\R^2),
\label{eq:bubble-weak-limit}
\\
U(t_n) &\to u_\ast
\qquad\text{strongly in }H^1(K_\rho),
\label{eq:bubble-strong-away}
\\
x_{j,n} &\to a_{\ell(j)},
\label{eq:bubble-center-limit}
\\
U(t_n,x_{j,n}+\lambda_{j,n}y)
&\to \omega_j(y)
\qquad\text{strongly in }C^1_{\mathrm{loc}}(\R^2),
\label{eq:bubble-rescaled}
\\
\frac{\lambda_{i,n}}{\lambda_{j,n}}
+\frac{\lambda_{j,n}}{\lambda_{i,n}}
+\frac{|x_{i,n}-x_{j,n}|}{\lambda_{i,n}+\lambda_{j,n}}
&\longrightarrow \infty
\qquad (i\neq j),
\label{eq:bubble-orthogonality}
\\
E\bigl(U(t_n)\bigr)
&\longrightarrow
E(u_\ast)+\sum_{j=1}^J E(\omega_j).
\label{eq:bubble-energy-split}
\end{align}
\end{theorem}

\begin{proof}
This is the direct specialization to the target manifold $N=\Sph$ of the Struwe-type bubble decomposition at finite singularity times for the two-dimensional Landau--Lifshitz flow in \cite{LiZhao}; the outputs \eqref{eq:bubble-weak-limit}--\eqref{eq:bubble-energy-split} are exactly all that will be invoked below.\qedhere
\end{proof}

\begin{proposition}[Strong endpoint singularity]\label{prop:maximal-strong-singular}
Let $u$ be a maximal strong solution of \eqref{eq:LLG-intro}, and suppose
\[
0<\Tstar<\infty.
\]
Then for every $T<\Tstar$,
\[
u\in \mathcal W([0,T]\times\R^2),
\qquad
u\in C^\infty((0,T)\times\R^2;\Sph),
\]
and $\Tstar$ is a finite singularity time in the sense above.
\end{proposition}

\begin{proof}
By Lemma~\ref{lem:strong-in-weakclass}, for every $T<\Tstar$,
\[
u\in \mathcal W([0,T]\times\R^2).
\]
Moreover, since
\[
u-Q\in C([0,\Tstar);H^2)\cap L^2_{\mathrm{loc}}([0,\Tstar);H^3),
\qquad
\partial_tu\in L^2_{\mathrm{loc}}([0,\Tstar);H^1),
\]
substitution into the divergence form \eqref{eq:LLG-div} shows that the right-hand side belongs to $L^2_{\mathrm{loc}}([0,\Tstar);H^1)$. Standard quasilinear parabolic regularization and bootstrap then give
\[
u\in C^\infty((0,\Tstar)\times\R^2;\Sph).
\]

It remains to prove that $\Tstar$ is indeed a finite singularity time. Suppose the contrary. Then there exist $\delta>0$ and
\[
\widetilde u\in \mathcal W([0,\Tstar+\delta)\times\R^2)
\cap C^\infty((0,\Tstar+\delta)\times\R^2;\Sph)
\]
such that $\widetilde u=u$ on $(0,\Tstar)\times\R^2$. By positive-time parabolic regularization, on every compact time interval $[s_1,s_2]\Subset(0,\Tstar+\delta)$ one has
\[
\widetilde u-Q\in C([s_1,s_2];H^2(\R^2;\R^3)),
\qquad
\sup_{s\in[s_1,s_2]}\norm{\widetilde u(s)-Q}_{H^2}<\infty.
\]
Choose $\varepsilon\in(0,\delta/4)$ and set $t_0=\Tstar-\varepsilon$. Then
\[
\widetilde u(t_0)=u(t_0)\in H_Q^2(\R^2;\Sph),
\qquad
\sup_{s\in[t_0,\Tstar+\varepsilon]}\norm{\widetilde u(s)-Q}_{H^2}<\infty.
\]
The local strong-solution theory gives a unique $H_Q^2$ strong solution $v$ with initial value $u(t_0)$, defined on some interval $[t_0,t_0+\eta)$; the boundedness above allows the standard continuation argument to extend $v$ along $\widetilde u$ step by step, at least until $\Tstar+\varepsilon$. On the other hand, $u$ and $\widetilde u$ coincide on $(t_0,\Tstar)$, and $v$ and $\widetilde u$ have the same initial value. Hence, by uniqueness of strong solutions,
\[
v=\widetilde u=u
\qquad\text{on }[t_0,\Tstar)\times\R^2,
\]
and $v$ continues $u$ beyond $\Tstar$. This contradicts the maximality of $\Tstar$. Therefore $\Tstar$ must be a finite singularity time.
\end{proof}

\begin{lemma}[Degree splitting]\label{lem:degree-splitting}
In the notation of Theorem~\ref{thm:bubble-degree-input}, set
\[
d_j:=\operatorname{deg}(\omega_j)\in \mathbb Z\setminus\{0\}
\qquad (1\le j\le J).
\]
Then
\begin{equation}\label{eq:bubble-degree-split}
\operatorname{deg}(u_0)
=
\operatorname{deg}(u_\ast)+\sum_{j=1}^J d_j,
\end{equation}
and
\begin{equation}\label{eq:bubble-energy-lower}
\lim_{n\to\infty}E\bigl(U(t_n)\bigr)
\ge
4\pi\,\bigl|\operatorname{deg}(u_\ast)\bigr|
+
4\pi\sum_{j=1}^J |d_j|.
\end{equation}
\end{lemma}

\begin{proof}
Write
\[
u_n:=U(t_n),
\qquad
\mathcal J(v):=v\cdot(\partial_1 v\times \partial_2 v).
\]
By Lemma~\ref{lem:degree-energy}, for every $n$,
\[
\operatorname{deg}(u_n)=\operatorname{deg}(u_0),
\qquad
|\mathcal J(v)|\le \frac12|\nabla v|^2.
\]

\smallskip
\noindent Step 1: convergence of Jacobians away from singular points.
Fix $0<\rho<\rho_0$ and set
\[
K_\rho:=\R^2\setminus \bigcup_{\alpha=1}^N B(a_\alpha,\rho).
\]
By \eqref{eq:bubble-strong-away},
\[
u_n-u_\ast\to 0\qquad\text{strongly in }H^1(K_\rho).
\]
Therefore
\begin{align}
\bigl\|\partial_1 u_n\times \partial_2 u_n-\partial_1 u_\ast\times \partial_2 u_\ast\bigr\|_{L^1(K_\rho)}
&\le \norm{\partial_1(u_n-u_\ast)}_{L^2(K_\rho)}\norm{\partial_2u_n}_{L^2(K_\rho)}
\nonumber\\
&\quad+\norm{\partial_1u_\ast}_{L^2(K_\rho)}\norm{\partial_2(u_n-u_\ast)}_{L^2(K_\rho)}
\to0.
\label{eq:cross-away-conv}
\end{align}
Thus the sequence
\[
g_n:=\bigl|\partial_1 u_n\times \partial_2 u_n\bigr|
\]
is uniformly integrable in $L^1(K_\rho)$. Since also $u_n\to u_\ast$ in $H^1(K_\rho)$, after passing to a subsequence we may assume $u_n\to u_\ast$ a.e. on $K_\rho$; Vitali's theorem then gives
\begin{equation*}
\int_{K_\rho}|u_n-u_\ast|\,\bigl|\partial_1 u_n\times \partial_2 u_n\bigr|\,dx\to0.
\end{equation*}
Combining this with \eqref{eq:cross-away-conv}, we get
\begin{align}
\int_{K_\rho}\mathcal J(u_n)\,dx-\int_{K_\rho}\mathcal J(u_\ast)\,dx
={}&
\int_{K_\rho}(u_n-u_\ast)\cdot(\partial_1u_n\times\partial_2u_n)\,dx
\nonumber\\
&+\int_{K_\rho}u_\ast\cdot\Bigl[(\partial_1u_n\times\partial_2u_n)-(\partial_1u_\ast\times\partial_2u_\ast)\Bigr]dx
\nonumber\\
\longrightarrow{}0.
\label{eq:jacobian-away}
\end{align}

\smallskip
\noindent Step 2: diagonal extraction of bubble balls.
Fix the above $\rho$. Choose an increasing sequence $R_m\uparrow\infty$ such that, for every $j$,
\begin{equation}\label{eq:Rm-tail-choice}
E\bigl(\omega_j;\R^2\setminus B(0,R_m)\bigr)
+
\int_{\R^2\setminus B(0,R_m)}|\mathcal J(\omega_j)|\,dy
\le 2^{-m}.
\end{equation}
By \eqref{eq:bubble-rescaled}, for each $m$ one can choose $N_m$ such that, whenever $n\ge N_m$,
\begin{equation}\label{eq:rescaled-C1-choice}
\sum_{j=1}^J
\norm{u_n(x_{j,n}+\lambda_{j,n}\cdot)-\omega_j}_{C^1(B(0,R_m))}
\le 2^{-m}.
\end{equation}
Using \eqref{eq:bubble-center-limit}--\eqref{eq:bubble-orthogonality} and increasing $N_m$ if necessary, we may also require that for all $n\ge N_m$,
\begin{equation*}
B(x_{j,n},R_m\lambda_{j,n})\subset B(a_{\ell(j)},\rho),
\qquad
B(x_{i,n},R_m\lambda_{i,n})\cap B(x_{j,n},R_m\lambda_{j,n})=\varnothing\qquad (i\ne j).
\end{equation*}
Take an increasing subsequence $n_m\ge N_m$ and relabel it by $n$. Put
\[
R_n:=R_m,
\qquad
B_{j,n}:=B(x_{j,n},R_n\lambda_{j,n}).
\]
Then $R_n\uparrow\infty$, and \eqref{eq:rescaled-C1-choice}, \eqref{eq:Rm-tail-choice} imply
\begin{align}
E(u_n;B_{j,n})
&=E\bigl(u_n(x_{j,n}+\lambda_{j,n}\cdot);B(0,R_n)\bigr)
\longrightarrow E(\omega_j),
\label{eq:bubble-ball-energy-conv}
\\
\int_{B_{j,n}}\mathcal J(u_n)\,dx
&=
\int_{B(0,R_n)}\mathcal J\bigl(u_n(x_{j,n}+\lambda_{j,n}y)\bigr)\,dy
\longrightarrow
\int_{\R^2}\mathcal J(\omega_j)\,dy
=
4\pi d_j.
\label{eq:jacobian-bubble}
\end{align}

\smallskip
\noindent Step 3: smallness of the neck energy.
Define the neck region
\[
\Gamma_{n,\rho}
:=
\Bigl(\bigcup_{\alpha=1}^N B(a_\alpha,\rho)\Bigr)
\setminus
\Bigl(\bigcup_{j=1}^J B_{j,n}\Bigr).
\]
By disjointness and the decomposition,
\begin{equation*}
E(u_n)
=
E(u_n;K_\rho)
+
\sum_{j=1}^J E(u_n;B_{j,n})
+
E(u_n;\Gamma_{n,\rho}).
\end{equation*}
Using \eqref{eq:bubble-strong-away}, \eqref{eq:bubble-ball-energy-conv}, and \eqref{eq:bubble-energy-split}, we obtain
\begin{equation*}
\lim_{n\to\infty}E(u_n;\Gamma_{n,\rho})
=
E\Bigl(u_\ast;\bigcup_{\alpha=1}^N B(a_\alpha,\rho)\Bigr).
\end{equation*}
Consequently,
\begin{equation}\label{eq:neck-jacobian-bound}
\limsup_{n\to\infty}
\left|
\int_{\Gamma_{n,\rho}}\mathcal J(u_n)\,dx
\right|
\le
E\Bigl(u_\ast;\bigcup_{\alpha=1}^N B(a_\alpha,\rho)\Bigr).
\end{equation}

\smallskip
\noindent Step 4: splitting of the degree.
By the region decomposition,
\begin{align}
4\pi\,\operatorname{deg}(u_0)
&=4\pi\,\operatorname{deg}(u_n)
=\int_{\R^2}\mathcal J(u_n)\,dx
\nonumber\\
&=
\int_{K_\rho}\mathcal J(u_n)\,dx
+
\sum_{j=1}^J\int_{B_{j,n}}\mathcal J(u_n)\,dx
+
\int_{\Gamma_{n,\rho}}\mathcal J(u_n)\,dx.
\notag
\end{align}
Letting $n\to\infty$ and using \eqref{eq:jacobian-away}, \eqref{eq:jacobian-bubble}, and \eqref{eq:neck-jacobian-bound}, we get
\begin{equation*}
4\pi\,\operatorname{deg}(u_0)
=
\int_{K_\rho}\mathcal J(u_\ast)\,dx
+
4\pi\sum_{j=1}^J d_j
+
\varepsilon_\rho,
\qquad
|\varepsilon_\rho|
\le
E\Bigl(u_\ast;\bigcup_{\alpha=1}^N B(a_\alpha,\rho)\Bigr).
\end{equation*}
Finally let $\rho\downarrow0$. Since $u_\ast\in H_Q^1(\R^2;\Sph)$,
\[
\int_{K_\rho}\mathcal J(u_\ast)\,dx
\longrightarrow
\int_{\R^2}\mathcal J(u_\ast)\,dx
=4\pi\operatorname{deg}(u_\ast),
\qquad
E\Bigl(u_\ast;\bigcup_{\alpha=1}^N B(a_\alpha,\rho)\Bigr)\to0.
\]
Thus \eqref{eq:bubble-degree-split} holds.

\smallskip
\noindent Step 5: the energy lower bound.
By \eqref{eq:bubble-energy-split} and Lemma~\ref{lem:degree-energy},
\[
\lim_{n\to\infty}E(u_n)
=
E(u_\ast)+\sum_{j=1}^J E(\omega_j)
\ge
4\pi\,\bigl|\operatorname{deg}(u_\ast)\bigr|
+
4\pi\sum_{j=1}^J |d_j|,
\]
which is \eqref{eq:bubble-energy-lower}.
\end{proof}

With the classical profile decomposition, the singularity bridge for endpoints of maximal strong solutions, and the degree splitting derived in the body of the paper in hand, the strong-solution framework of the first four sections can be spliced directly with the geometric quantization of bubble trees. The key point is that, in the degree-zero sector, the minimal energy cost of any nontrivial bubble tree is not $4\pi$, but $8\pi$.

\begin{proposition}[$8\pi$ obstruction]\label{prop:deg0-8pi-barrier}
Let $u$ be a maximal strong solution of \eqref{eq:LLG-intro}, and assume
\[
0<\Tstar<\infty,
\qquad
\operatorname{deg}(u_0)=0.
\]
Then necessarily
\begin{equation}\label{eq:8pi-barrier}
E(u_0)\ge 8\pi.
\end{equation}
\end{proposition}

\begin{proof}
By Lemma \ref{lem:degree-energy}, for every $t\in[0,\Tstar)$,
\begin{equation}\label{eq:deg-const-before-T}
\operatorname{deg}(u(t))=\operatorname{deg}(u_0)=0.
\end{equation}
On the other hand, Proposition~\ref{prop:maximal-strong-singular} implies that $u$ is smooth on $(0,\Tstar)$ and that $\Tstar$ is a finite singularity time in the sense of Theorem~\ref{thm:bubble-degree-input}. Thus one may choose $t_n\uparrow \Tstar$, a limit map $u_\ast$, and harmonic bubbles $\omega_1,\dots,\omega_J$ such that \eqref{eq:bubble-energy-split} holds. Set
\[
d_j:=\operatorname{deg}(\omega_j)\in\mathbb Z\setminus\{0\}.
\]
By Lemma \ref{lem:degree-splitting} and \eqref{eq:deg-const-before-T},
\begin{equation*}
\operatorname{deg}(u_\ast)=-\sum_{j=1}^J d_j.
\end{equation*}
Using energy monotonicity and \eqref{eq:bubble-energy-lower} in Lemma \ref{lem:degree-splitting}, we get
\begin{align}
E(u_0)
&\ge \lim_{n\to\infty}E\bigl(u(t_n)\bigr)
\label{eq:energy-lower-deg0}
\\
&\ge
4\pi\,\bigl|\operatorname{deg}(u_\ast)\bigr|
+
4\pi\sum_{j=1}^J |d_j|
\nonumber\\
&=
4\pi\Bigl(\Bigl|\sum_{j=1}^J d_j\Bigr|+\sum_{j=1}^J |d_j|\Bigr).
\nonumber
\end{align}
Since every $d_j\neq0$ and $J\ge1$,
\begin{equation}\label{eq:integer-lower-bound}
\Bigl|\sum_{j=1}^J d_j\Bigr|+\sum_{j=1}^J |d_j|\ge 2.
\end{equation}
Indeed, if $\sum_j d_j\neq0$, then the left-hand side is at least $1+1=2$; if $\sum_j d_j=0$, then because not all $d_j$ are zero, necessarily $\sum_j |d_j|\ge2$. Substituting \eqref{eq:integer-lower-bound} into \eqref{eq:energy-lower-deg0} gives
\[
E(u_0)\ge 8\pi,
\]
which is \eqref{eq:8pi-barrier}.
\end{proof}

\begin{corollary}[Degree-zero divergence]\label{cor:deg0-analytic-div}
Let $u$ be a maximal strong solution of \eqref{eq:LLG-intro}, and assume
\[
0<\Tstar<\infty,\qquad \deg(u_0)=0.
\]
Then, besides \eqref{eq:8pi-barrier}, for every $T_0\in[0,\Tstar)$ and every dyadic index $k_\ast\in\mathbb Z$,
\begin{align}
\int_{T_0}^{\Tstar}\int_{\R^2}|\nabla u(t,x)|^4\,dxdt&=\infty,\label{eq:deg0-L4-div}\\
\norm{P_{\ge k_\ast}\nabla u}_{L^4_{t,x}((T_0,\Tstar)\times\R^2)}&=\infty.\label{eq:deg0-uv-div}
\end{align}
\end{corollary}

\begin{proof}
Inequality \eqref{eq:8pi-barrier} has already been proved in Proposition~\ref{prop:deg0-8pi-barrier}. If for some $T_0<\Tstar$ one had
\[
\int_{T_0}^{\Tstar}\int_{\R^2}|\nabla u|^4<\infty,
\]
then Theorem~\ref{thm:L4-cont} would immediately imply that $u$ can be continued beyond $t=\Tstar$, contradicting maximality. Hence \eqref{eq:deg0-L4-div} holds. Similarly, if for some $k_\ast$ one had
\[
\norm{P_{\ge k_\ast}\nabla u}_{L^4_{t,x}((T_0,\Tstar)\times\R^2)}<\infty,
\]
then Theorem~\ref{thm:uv-tail-cont} would again give a continuation contradiction, and therefore \eqref{eq:deg0-uv-div} holds.
\end{proof}

The proposition above gives the core geometric obstruction desired in this section: in the degree-zero sector, finite-time singularities can occur only starting at $8\pi$; Corollary~\ref{cor:deg0-analytic-div} then genuinely splices this geometric obstruction with the continuation mechanism established in the first four sections. We therefore conclude this section with the proof of Theorem~\ref{thm:deg0-8pi-global} from the introduction: when the degree is zero and the energy is below $8\pi$, the maximal strong solution cannot terminate in finite time.

\begin{proof}[Proof of Theorem \ref{thm:deg0-8pi-global}]
Assume the contrary, namely
\[
0<\Tstar<\infty.
\]
Then Corollary~\ref{cor:deg0-analytic-div} gives
\[
E(u_0)\ge 8\pi,
\qquad
\int_{T_0}^{\Tstar}\int_{\R^2}|\nabla u|^4=\infty,
\qquad
\norm{P_{\ge k_\ast}\nabla u}_{L^4_{t,x}((T_0,\Tstar)\times\R^2)}=\infty
\]
for every $T_0<\Tstar$ and every $k_\ast\in\mathbb Z$. The first of these statements already contradicts the assumption \eqref{eq:deg0-8pi-main-ass}. Thus a finite terminal time cannot occur, and hence
\[
\Tstar=\infty.
\]
The stated local regularity of the global strong solution is exactly the continuation of the definition of maximal strong solution on arbitrary finite time intervals.
\end{proof}

\begin{remark}
The relation between the result of this section and the first four sections can be summarized as follows: the first four sections provide the analytic continuation mechanism at a finite terminal time, while this section shows that, in the degree-zero sector with energy below $8\pi$, the geometric carrier that would trigger a finite-time singularity--a nontrivial bubble tree--does not exist. Therefore the $8\pi$ improvement is not the product of a new $H^2$ differential inequality, but the combination of the strong-solution framework of this paper with bubble quantization under the degree-zero constraint.
\end{remark}

\begin{remark}
The degree-zero hypothesis is not a superficial decoration. Known $1$-equivariant finite-time blow-up is built around the lowest-energy nonconstant steady state. Consequently, one cannot expect a uniform general threshold, valid across all topological sectors, that is strictly larger than $4\pi$. In contrast, the key to the improvement to $8\pi$ in this section is that degree zero rules out the cheapest singularity mechanism consisting of ``one bubble core of degree $\pm1$ plus a constant background.'' For a related $1$-equivariant blow-up construction, see \cite{XuZhao}.
\end{remark}

\section*{Funding}
Pang-hung Chung was supported by the Project ``Research on Nonlinear Partial Differential Equations'' (No. 2024KYCXTD018), Special Projects in Key Areas of Guangdong Province (No. ZDZX1088), and the Fund for Science and Technology of Guangzhou (No. 202102080428).

\section*{Author contributions}
Pang-hung Chung and Dan Han contributed to the conception, mathematical analysis, verification, revision, and preparation of the manuscript. Both authors read and approved the final manuscript.

\section*{Data availability}
Data sharing is not applicable to this article as no datasets were generated or analyzed during the current study.

\section*{Conflict of interest}
The authors declare that they have no conflict of interest.

\end{document}